\documentclass[12pt]{amsart}
\usepackage{amsmath,amssymb,amsfonts,amsthm,amsopn}

\usepackage{mathrsfs}
\usepackage{graphicx,tikz}
\usepackage[all]{xy}
\usepackage{amsmath}
\usepackage{multirow, longtable, makecell, caption, array}
\usepackage{amssymb}
\usepackage{mathtools}
\usepackage{pb-diagram}
\usepackage{cellspace} %
\newtheorem{theorem}{Theorem}[section]
\newtheorem{corollary}[theorem]{Corollary}
\newtheorem{lemma}[theorem]{Lemma}
\newtheorem{definition}[theorem]{Definition}

\newtheorem{proposition}[theorem]{Proposition}
\newtheorem{remark}[theorem]{Remark}
\newcommand{\beqa}{\begin{eqnarray*}}
	\newcommand{\eeqa}{\end{eqnarray*}}

\newcommand{\field}[1]{\mathbb{#1}}
\newcommand{\bR}{\field{R}}
\newcommand{\bN}{\field{N}}

\def\al{\alpha}
\def\be{\beta}

\def\la{\lambda}

\def\om{\omega}

\def\cF{\mathcal{F}}
\def\cS{\mathcal{S}}

\def\cG{\mathcal{G}}

\def\cU{\mathcal{U}}

\def\cJ{\mathcal{J}}

\def\tW{\widetilde{W}}

\def\rd{\bR^d}

\def\rdd{{\bR^{2d}}}

\def\intrd{\int_{\rd}}
\def\intrdd{\int_{\rdd}}

\def\<{\left<}
\def\>{\right>}

\def\mv1{M_v^1}

\newcommand{\abs}[1]{\left|#1\right|}
\newcommand{\norm}[1]{\left\|#1\right\|}

\def\Ren{\mathbb{R}^d}

\def\f{\varphi}

\def\Sn2{S_{2}(L^{2}(\Ren))}
\def\S1{S_{1}(L^{2}(\Ren))}
\def\sig00{\sigma_{0,0}}

\def\la{\langle}
\def\ra{\rangle}

\usepackage{pifont}

\DeclareMathOperator*{\Opt}{Op_{\tau}}

\usepackage{upgreek}
\usepackage[]{bm}

\usepackage{hyperref}

\DeclareMathOperator*{\tr}{tr}

\def\FO{\mathbb{S}_{0}} 
\def\MO{\mathbb{S}^\prime_{0}} 
\def\FA{\mathcal{S}_{0}} 
\def\MD{\mathcal{S}^\prime_{0}} 
\def\tTFS{\pi^\tau} 
\def\omtTFS{\pi^{1-\tau}} 
\def\zTFS{\pi^0} 
\newcommand{\dual}[2]{\la #1 {,}#2 \ra} 
\newcommand{\dualOP}[2]{\la #1 {,}#2 \ra} 
\def\TC{\cJ^1} 
\def\HS{\cJ^2} 
\def\BO{\cJ^\infty} 
\def\tV{V^\tau}
\def\zV{V^0} 

\def\ohV{V^{\frac12}} 
\newcommand{\K}[1]{K_{#1}}
\def\tW{W_\tau}
\def\omtW{W_{1-\tau}}
\def\ohW{W_{1/2}}
\def\FtW{\mathcal{F}_{\tW}}
\def\FohW{\cF_{W_{1/2}}}
\def\FomtW{\cF_{W_{1-\tau}}}
\def\Sch{\mathcal{S}}
\def\Tem{\Sch'}
\newcommand{\tsym}[1]{\mathrm{a}^{#1}_\tau}
\newcommand{\seqj}[1]{\{#1_j\}_j}
\newcommand{\seqn}[1]{\{#1_n\}_n}
\newcommand{\sumj}{\sum_{j=1}^\infty}
\newcommand{\sumk}{\sum_{k=1}^\infty}
\newcommand{\sumn}{\sum_{n=1}^\infty}
\newcommand{\normFA}[1]{\norm{#1}_{\FA}}
\newcommand{\normFO}[1]{\norm{#1}_{\mathbb{S}_0}}
\newcommand{\normMO}[1]{\norm{#1}_{\mathbb{S}'_0}}

\newcommand{\netn}[1]{\{#1_n\}_n}
\newcommand{\limnetn}{\lim_{n\to+\infty}}
\newcommand{\wsconvMOn}[2]{#1_n\overset{w-\ast}{\underset{n}{\longrightarrow}}#2\quad\text{in}\quad\MO}
\newcommand{\wsconvMDn}[2]{#1_n\overset{w-\ast}{\underset{n}{\longrightarrow}}#2\quad\text{in}\quad\MD(\rdd)}

\DeclareMathOperator*{\wslimMOn}{\textit{w}-\ast-lim_\textit{n}}
\DeclareMathOperator*{\SR}{SR}
\def\tSR{{\SR}^\tau}
\newcommand{\tQ}[2]{Q^\tau_{#1}(#2)}
\DeclareMathOperator*{\Opomt}{Op_{1-\tau}}
\newcommand{\GM}[1]{G^\f_{#1}}
\def\twistmatrix{\Theta}
\DeclareMathOperator{\twistconv}{\natural_\twistmatrix}

\def\wFO{\mathbb{M}_s^1}
\def\wFA{M^1_{v_s\otimes v_s}}

\begin{document}
	\begin{abstract} 
	We investigate the $\tau$-quantizations and Cohen's class distributions of a suitable class of trace-class operators, called Feichtinger's operators, and show that it is a convenient substitute for the class of Schwartz operators. Many well-known concepts and results for functions in time-frequency analysis have an operator-analog in our setting, e.g. that Cohen's classes are convolutions of Wigner functions with distributions or characterization of the class of Schwartz operators as an intersection of weighted variants of the class of Feichtinger operators. 
	\end{abstract}
	
	\title[$\tau$-quantization and $\tau$-Cohen classes of Feichtinger operators]{$\tau$-quantization and $\tau$-Cohen classes distributions  of Feichtinger operators}
	\author{Federico Bastianoni}
	\address{Dipartimento di Scienze Matematiche, Politecnico di Torino, corso
		Duca degli Abruzzi 24, 10129 Torino, Italy}
	\email{federico.bastianoni@polito.it}
	\thanks{}
	\author{Franz Luef}
	\address{Department of Mathematics, NTNU Norwegian University of Science and Technology, NO-7491 Trondheim, Norway
	}
	\email{franz.luef@ntnu.no}
	\thanks{}

	\date{}
\maketitle

\tableofcontents

\section{Introduction}

There is a vast literature on the boundedness of pseudodifferential operators for certain classes of symbols in various quantization schemes along the lines of H\"ormander classes or alternatively using Sj\"ostrand's class or Shubin's classes, e.g. \cite{BasCorNic20,EleCharly2003,Wignersharp2018,GroStr2007,ToftquasiBanach2017}. In the present work, we put our focus on Shubin's $\tau$-quantization and the associated time-frequency representations, the $\tau$-Cohen classes. 

Our approach to this circle of ideas is based on the framework of quantum harmonic analysis with the goal to lift the well-known results concerning functions to an appropriate class of functions, which we call {\it Feichtinger operators}, $\FO$, and which is the operator analog of the well-known Feichtinger algebra $\FA$. 

We also discuss the relation between Feichtinger operators $\FO$ and the class of Schwartz operators introduced by Keyl, Kiukas and Werner in \cite{KeyKiuWer2016}. There the idea is put forward that one should look for analogs of function spaces in the setting of classes of operators, which has been realized in the case of Sobolev spaces in \cite{Laf2022} and for modulation spaces in \cite{DoeLueMcnSkr2022}.

For $\tau\in[0,1]$ the $\tau$-quantization of a symbol $a\in\Tem(\rdd)$, the space of tempered distributions, is given by 
\begin{equation}\label{Eq-Def-Opt-formal-integral}
	\Opt (a)f(t)\coloneqq\intrdd e^{2\pi i(t-y)\xi}a((1-\tau)t+\tau y,\xi)f(y)\,dy d\xi~~f\in\Sch(\rd),
\end{equation}
where the operator $\Opt(a)$ is understood to be defined in the weak sense. A well-known fact is that one can relate $\la \Opt(a)f,g\ra$ to a time-frequency representation, $\tW(f,g)$, the cross-$\tau$-Wigner distribution of $f$ and $g$:  
\begin{equation*}
	\la \Opt(a)f,g\ra=\la a,\tW(g,f)\ra,\qquad\text{for all}\quad f,g\in\Sch(\rd).
\end{equation*}

Given an operator $S$, we denote by $\tsym{S}$ its $\tau$-symbol, i.e. the tempered distribution such that $\Opt\left(\tsym{S}\right)=S$ and $\Opt$ is called the $\tau$-Shubin quantization.  For $f,g\in L^2(\rd)$ we denote the rank-one operator by $f\otimes g$ and note that $\tsym{f\otimes g}=W_\tau(g,f)$, i.e. there is an intrinsic relation between quantization schemes and time-frequency representations. 

We show that for well-behaved operators, e.g. trace class operators or Feichtinger operators, this relation might be extended to operators. Recall that Wigner in his ground-breaking work on quasi-probability distributions introduced the cross-Wigner distribution for certain classes of operators \cite{Wigner1932}, which was later extended to more general classes of operators by Moyal in \cite{Moy1949}. 

Let $S$ be a continuous operator between the Feichtinger algebra $\FA$ and its continuous dual space $\MD$. We denote by $\K{S}$ the kernel of $S$, which exists by Feichtinger's kernel theorem and is a mild distribution on $\mathbb{R}^{2d}$. 

We define Feichtinger operators, $\FO$, to be the following class of continuous and linear operators $\FO\coloneqq S\colon \MD(\rd)\to\FA(\rd)$ that map norm bounded w-$\ast$ convergent sequences in $\MD$ into norm convergent sequences in $\FA$. In \cite{FeiJak2022} it was shown that these are precisely the linear continuous operators from $\MD$ to $\FA$ that have a kernel in Feichtinger's algebra, the so-called inner kernel theorem. 

One of our main tools is that Feichtinger operators have a nice spectral decomposition. If $S$ is in $\FO$, then there exist two (non-unique) sequences $\seqn{f},\seqn{g}\subseteq\FA(\rd)$ such that
	\begin{equation*}
		S=\sumn f_n\otimes g_n,\qquad\sumn\normFA{f_n}\normFA{g_n}<\infty, \qquad\K{S}=\sumn\K{f_n\otimes g_n}.
	\end{equation*}
Hence, Feichtinger operators are trace class operators and we can compute their trace as follows $\tr(S)=\intrd\K{S}(x,x)\,dx$. In \cite{GosGos2021} operators having such a decomposition have been studied and called Feichtinger states in case $\tr(S)=1$, but there the link between these operators and the work \cite{FeiJak2022} was not established, which is one of our main observations.

Then the $\tau$-Wigner distribution of $S$ is defined in the following way
\begin{equation}\label{Eq-Def-tauWigner-operators}
	\tW S(x,\om)\coloneqq\intrd e^{-2\pi i t\om}\K{S}(x+\tau t,x-(1-\tau)t)\,dt.
\end{equation}
Our key observation is the following identity:
\begin{equation*}
\dual{a}{\tW S}=\mathrm{tr}(\Opt(a)S^*)=:\dualOP{\Opt(a)}{S},    
\end{equation*}
for $S$ in $\FO$ or $\TC$, and $\tW S$ is the $\tau$-Wigner distribution of $S$. Consequently, we interpret $\tW S$ as the $\tau$-quantization of an operator in $\FO$ or $\TC$. 

Note, that if $S$ is the rank-one operator $f\otimes g$ this becomes the aforementioned relation between the $\tau$-Wigner distribution and the Shubin $\tau$-transform.

Based on this framework we deduce operator analogs of well-known results on $\tau$-Wigner distributions and $\tau$-Shubin quantization, which indicates that this is a very convenient setting for this type of investigation. In addition, we extend the Cohen class of an operator, introduced in \cite{Luef1}, to the $\tau$-setting and show that it can be written as the convolution of the Wigner distribution of an operator with a distribution as in the function setting. 

We close our discussion with the introduction of weighted versions of $\FO$ and prove that the intersection of all these is the class of Schwartz operators in \cite{KeyKiuWer2016}. As in the case of functions, we hope that this global description of the Schwartz operators will also turn out to be useful in subsequent studies and it also hints at operator analogs of Gelfand-Shilov classes or other classes of test functions and the corresponding class of ultradistributions.

\section{Preliminaries}\label{Sec-Preliminaries}
In this paper, the parameter $\tau$ always belongs to $[0,1]$, even when not specified.

\subsection{A family of time-frequency representations}

For $x,\om\in\rd$ we define the translation and modulation operator by
\begin{equation*}
	T_x f(t)\coloneqq f(t-x),\qquad M_\om f(t)\coloneqq e^{2\pi i \om t}f(t),\qquad \forall t\in\rd,
\end{equation*}
respectively. Their composition is denoted by $\pi(x,\om)\coloneqq M_\om T_x$.

Given $\tau\in[0,1]$, the $\tau$-time-frequency shift ($\tau$-TFS) at $(x,\om)\in\rdd$ is defined to be
\begin{equation}
	\tTFS(x,\om)\coloneqq e^{-2\pi i \tau x\om}M_\om T_x= M_{(1-\tau)\om} T_x M_{\tau\om}.
\end{equation}
For $\tau=0$ we recover the usual time-frequency shifts $\zTFS=\pi$.
The following relations are consequences of elementary computations, which are left to the reader:
\begin{align*}
	\tTFS(x,\om)\tTFS(x',\om')&=e^{-2\pi i[(1-\tau)x\om'-\tau x'\om]}\tTFS(x+x',\om+\om'),\\
	\tTFS(x,\om)\tTFS(x',\om')&=e^{-2\pi i[x\om'-x'\om]}\tTFS(x',\om')\tTFS(x,\om),\\
	\tTFS(x,\om)^\ast&=\omtTFS(-x,-\om)=e^{-2\pi i(1-\tau)x\om}\pi(-x,-\om).
\end{align*}
In the present paper the symbol $\la\cdot{,}\cdot\ra$ either denotes the inner product in $L^2(\rd)$ or a duality pairing between a Banach space $X$ and its dual space $X^\prime$, which is compatible with the latter, i.e.  $\la\cdot{,}\cdot\ra$ is assumed to be linear in the first argument and conjugate-linear in the second one. In particular, the dual pairs considered in this work are $(L^2,L^2),\,(\MD,\FA),\,(\MO,\FO)$, respectively. 

Above, $\FA$ is the Feichtinger algebra \eqref{Eq-Def-FA}, for the definitions of $\FO$ and $\MO$ see the equations \eqref{Eq-Def-FO},\eqref{Eq-Def-MO} and \eqref{Eq-duality-FO-MO}. 
We introduce for $f,g\in L^2(\rd)$, or for any suitable dual pair, the $\tau$-short-time Fourier transform ($\tau$-STFT) of $f$ w.r.t $g$:
\begin{equation}
	\tV_g f(x,\om)\coloneqq\la f,\tTFS(x,\om)g\ra,\qquad\forall x,\om\in\rd.
\end{equation}
As can be easily verified, the mapping 
\begin{equation*}
	\tTFS\colon\rdd\to\cU(L^2(\rd)),
\end{equation*}
where $\cU(L^2(\rd))$ denotes the unitary operators on $L^2(\rd)$, is a projective representation of $\mathbb{R}^{2d}$ for any $\tau$. Consequently, $\tV$ is the wavelet transform associated to $\tTFS$, thus $\tV_g f$ is a continuous function.

\begin{remark}
	For $\tau=0$ we obtain the usual STFT $\zV_gf=V_gf$ and we have 
	\begin{equation}
		\tV_g f(x,\om)=e^{2\pi i \tau x\om} V_g f(x,\om).
	\end{equation}
	By the preceding identity, we have that $\ohV_g f$ is the cross-ambiguity function of $f$ and $g$:
	\begin{equation}
		\ohV_g f(x,\om)=A(f,g)(x,\om).
	\end{equation}
\end{remark}

We recall another frequently used time-frequency representation, the so-called cross-$\tau$-Wigner distribution of $f$ and $g$ in $L^2(\rd)$ defined by
\begin{equation}\label{Eq-Def-tauWigner-functions}
	\tW(f,g)(x,\om)\coloneqq \intrd e^{-2\pi i t\om}f(x+\tau t)\overline{g(x-(1-\tau)t)}\,dt.
\end{equation}
We aim to extend the definition of $\tW$ from functions to operators, see \eqref{Eq-Def-tauWigner-operators}.

\subsection{Basics of QHA and novel tools} \label{Subsec-tools-QHA}
In this subsection we introduce the basic definitions of quantum harmonic analysis (QHA) following the seminal work of Werner \cite{Werner1984}. 
\\
For $z\in\rdd$ and $A\in B(L^2(\rd))$ the translation of the operator $A$ by $z$ is 
\begin{equation}
	\alpha_z(A)\coloneqq \pi(z)A\pi(z)^\ast,
\end{equation}
which satisfies $\alpha_{z}\alpha_{z'}=\alpha_{z+z'}$. By the parity operator, we mean
\begin{equation}
	Pf(t)\coloneqq \check{f}(t)\coloneqq f(-t),
\end{equation}
for any $f\in L^2(\rd)$, which induces an involution of $A\in B(L^2(\rd))$:
\begin{equation}
	\check{A}\coloneqq PAP.
\end{equation}
We denote by $\TC$ the space of all trace class operators on $L^2(\rd)$. Given $a\in L^1(\rdd)$ and $S\in\TC$. The convolution between $a$ and $S$ is the operator
\begin{equation}
	a\star S\coloneqq S\star a\coloneqq \intrdd a(z)\alpha_z(S)\,dz,
\end{equation}
were the integral may be interpreted in the weak sense. For operators $S,T\in\TC$, their convolution is the function defined for every $z\in\rdd$ as
\begin{equation}
	S\star T(z)\coloneqq \tr\left(S\alpha_z(\check{T})\right).
\end{equation}
In this paper, we reserve the symbol $\otimes$ for rank-one operators. Namely, given $f,g\in L^2(\rd)$:
\begin{equation}
	(f\otimes g)\psi\coloneqq\la \psi,g\ra f,\qquad\forall\psi\in L^2(\rd).
\end{equation}
The kernel of an operator $S$ will always be denoted by $\K{S}$. Evidently, the kernel of the operator $f\otimes g$ is the tensor product \textit{of functions} $f(x)\overline{g(y)}$: 
\begin{equation*}
	(f\otimes g)\psi (t)=\la \psi,g\ra f(t)=\intrd f(t)\overline{g(x)}\psi(x)\,dx.
\end{equation*}
In the sequel we denote the tensor product of two functions by $f(x)g(y)$, we shall adopt the notation
\begin{equation}
	\K{f\otimes\overline{g}}(x,y)=f(x)\overline{g}(y).
\end{equation}
We now interpret \eqref{Eq-Def-tauWigner-functions} as the cross-$\tau$-Wigner distribution of the rank-one operator $f\otimes g$.
\\
Hence, it is natural to define the $\tau${\it -Wigner distribution} of an {\it operator} $S$ with kernel $\K{S}$ in the following way:
\begin{equation}\label{Eq-Def-tauWigner-operators}
	\tW S(x,\om)\coloneqq\intrd e^{-2\pi i t\om}\K{S}(x+\tau t,x-(1-\tau)t)\,dt.
\end{equation}
For $S\in\TC$ and $\tau\in[0,1]$, we define the Fourier-$\tau$-Wigner transform of $S$ to be:
\begin{equation}
	\FtW S(z)\coloneqq\tr\left(\tTFS(z)^\ast S\right),\qquad\forall z\in\rdd.
\end{equation}
For $\tau=1/2$ we recover the usual Fourier-Wigner transform \cite{Werner1984}.
\\
The $\tau$-spreading representation of $S\in B(L^2)$ is the decomposition 
\begin{equation}
	S=\intrdd h(z)\tTFS(z)\,dz,
\end{equation}
where the integral is understood in the weak sense. The function $h$ is called the $\tau$-spreading function of $S$. 
\\
In the following, we shall consider the $\tau$-spreading representation as a quantization scheme that assigns to a function an operator. Namely, $h\in L^1(\rdd)$ gets associated to the operator
\begin{equation}\label{Eq-Def-tSR-L1}
	\tSR (h)\coloneqq \intrdd h(z)\tTFS(z)\,dz.
\end{equation}
Let $\cF_\sigma$ denote the symplectic Fourier transform. In the following lemma we collect a number of important relations between these notions. The proofs are elementary computations and based on the spectral decomposition of the trace class operators $S$ and $T$ (\cite{SimonTrace}), which we leave to the interested reader.

\begin{lemma}\label{Lem-QHA-new-tools}
	Let $f,g,\in L^2(\rd)$, $S,T\in\TC$, $a\in L^1(\rdd)$ and $\tau\in[0,1]$. Then:
	\begin{itemize}
		\item[$(i)$]  $\cF_\sigma (\tW(f\otimes g))=\tV_g f$;
		\item[$(ii)$] $\FtW(f\otimes g)=\tV_gf$;
		\item[$(iii)$] $\tW S=\cF_\sigma\FtW S$;
		\item[$(iv)$] $\FtW S(x,\om)=e^{-2\pi i (1/2-\tau)x\om}\FohW S(x,\om)$;
		\item[$(v)$] $\cF_\sigma(S\star T)=\FtW S\cdot \FomtW T=\FomtW S\cdot \FtW T$;
		\item[$(vi)$] $\FtW(a\star S)=\cF_\sigma a\cdot \FtW S$;
		\item[$(vii)$] $\FtW S$ is the $\tau$-spreading function of $S$, i.e. $S=\intrdd \FtW S(z)\tTFS(z)\,dz$.
	\end{itemize}
\end{lemma}

We notice that if we consider the rank-one operator $S=f\otimes g$, then the assertions $(iii)$ and $(ii)$ of the previous lemma imply
\begin{equation}
	\tW(f,g)=\tW(f\otimes g)=\cF_\sigma\tV_g f.
\end{equation}

\subsection{$\tau$-quantization of functions}
The $\tau$-quantization of a symbol $a\in\Tem(\rdd)$, the space of tempered distributions, is formally given by
\begin{equation}\label{Eq-Def-Opt-formal-integral}
	\Opt (a)f(t)\coloneqq\intrdd e^{2\pi i(t-y)\xi}a((1-\tau)t+\tau y,\xi)f(y)\,dy d\xi,
\end{equation}
where $f\in\Sch(\rd)$. $\Opt(a)$ may be described rigorously in the weak sense: 
\begin{equation*}
	\la \Opt(a)f,g\ra=\la a,\tW(g,f)\ra,\qquad\forall f,g,\in\Sch(\rd).
\end{equation*}
Given an operator $S$, we denote by $\tsym{S}$ its $\tau$-symbol, i.e. the tempered distribution such that
\begin{equation*}
	\Opt\left(\tsym{S}\right)=S.
\end{equation*}

\begin{remark}
	Under suitable assumptions, for example $a\in L^1(\rdd)$, straightforward calculations give
	\begin{equation*}
		\Opt(a)=\intrdd \cF_\sigma a(z)\tTFS(z)\,dz,
	\end{equation*}
	and since also $\FtW \Opt(a)$ is the $\tau$-spreading function of $\Opt(a)$, we have 
	\begin{equation}
		a=\cF_\sigma\FtW\Opt(a).
	\end{equation}
	Hence, for $S\in\TC$ 
	\begin{equation}\label{Eq-Rem-tsym}
		\tsym{S}=\cF_\sigma\FtW S=\tW S.
	\end{equation}
\end{remark}

Given $a\in\MD(\rdd)$ and $f,g\in\FA(\rd)$, we recall the definition of cross-$\tau$-Cohen's class representation of $f$ and $g$, with kernel $a$:
\begin{equation}\label{Eq-Def-tau-Cohen-functions}
	\tQ{a}{f,g}\coloneqq a\ast \tW(f,g).
\end{equation}

\section{Feichtinger operators}\label{Section-Fei-Op} 
In this section we summarize some important results concerning a class of operators studied in \cite{FeiJak2022}. For such operators, introduced below, we adopt the name \textquotedblleft Feichtinger operators\textquotedblright\, for reasons which will become evident later.
\\
We recall that the Feichtinger algebra over $\rd$ \cite{Feichtinger-On-a-new-segal-alg-1981} is the Banach space
\begin{equation}\label{Eq-Def-FA}
	\FA(\rd)\coloneqq\{f\in L^2(\rd)\,|\, V_g f\in L^1(\rdd)\},
\end{equation}
for some $g\in L^2(\rd)\smallsetminus\{0\}$, endowed with the norm
\begin{equation*}
	\normFA{f}\coloneqq\norm{V_gf}_{L^1}=\intrdd \abs{V_gf(x,\om)}\,dxd\om.
\end{equation*}
We refer the reader to \cite{Jakobsen2018} for a detailed survey on $\FA(\rd)$. In this work, $\MD(\rd)$ denotes the conjugate-dual of $\FA(\rd)$. 
\begin{definition}
	The set of  \underline{Feichtinger operators} is defined to be
	\begin{align}
		\FO\coloneqq&\{S\colon \MD(\rd)\to\FA(\rd)\,|\,S\,\text{is linear, continuous and}\notag\\
		&\text{maps norm bounded w-$\ast$ convergent sequences in $\MD$}\label{Eq-Def-FO}\\
		&\text{into norm convergent sequences in $\FA$}\}.\notag
	\end{align}
\end{definition}
We adopt the following notation:
\begin{equation}\label{Eq-Def-MO}
	\MO\coloneqq B(\FA(\rd),\MD(\rd))
\end{equation}
and state the so called Outer Kernel Theorem \cite[Theorem 1.1]{FeiJak2022}:
\begin{theorem}\label{Th-OuterKernel}
	The Banach space $\MO$ is isomorphic to $\MD(\rdd)$ via the map $T\mapsto\K{T}$, where the relation between $T$ and its kernel $\K{T}$ is given by
	\begin{equation*}
		\dual{T f}{g}=\dual{\K{T}}{\K{g\otimes f}},\qquad\forall\,f,g,\in\FA(\rd).
	\end{equation*}
\end{theorem}

The following statement goes under the name of Inner Kernel Theorem. We present it in our setting. 
To this end, we introduce the following notation: given $\sigma,\nu\in\MD(\rd)$, we denote by $\nu\widetilde{\otimes}\overline{\sigma}$ the unique element of $\MD(\rdd)$ such that
\begin{equation*}
	\dual{\nu\widetilde{\otimes}\overline{\sigma}}{\K{\psi\otimes\overline{\f}}}=\dual{\nu}{\psi}\overline{\dual{\sigma}{\overline{\f}}},\qquad\forall\,\psi,\f\in\FA(\rd).
\end{equation*}
We refer the reader to \cite[Theorem 1.3]{FeiJak2022}, Lemma 3.1 and Corollary 3.10, too.
\begin{theorem}\label{Th-InnerKernel}
	The space of Feichtinger operators $\FO$ is a Banach space if endowed with the norm of $B(\MD,\FA)$ and it is naturally isomorphic as Banach space to $\FA(\rdd)$ through the map $T\mapsto \K{T}$, where the relation between $T$ and its kernel $\K{T}$ is given by
	\begin{equation*}
		\dual{\nu}{T \sigma}=\dual{\nu\widetilde{\otimes}\overline{\sigma} }{\K{T}},\qquad\forall\,\sigma,\nu,\in\MD(\rd).
	\end{equation*}
	Moreover, $\FO$ is Banach algebra under composition. If $S,T\in\FO$, then
	\begin{equation}\label{Eq-ker-composition}
		\K{S\circ T}(y,u)=\intrd\K{T}(y,t)\K{S}(t,u)\,dt.
	\end{equation}
\end{theorem}

By the above theorems \ref{Th-OuterKernel} and \ref{Th-InnerKernel}, $\MO$ is the (conjugate) topological dual of $\FO$ and the duality is given by
\begin{equation}\label{Eq-duality-FO-MO}
	\dualOP{T}{S}=\dual{\K{T}}{\K{S}}.
\end{equation}

\begin{lemma}\label{Lem-representation-FO}
	Suppose $S\in\FO$. Then there exist two non-unique sequences $\seqn{f},\seqn{g}\subseteq\FA(\rd)$ such that
	\begin{equation*}
		S=\sumn f_n\otimes g_n,\qquad\sumn\normFA{f_n}\normFA{g_n}<+\infty, \qquad\K{S}=\sumn\K{f_n\otimes g_n}.
	\end{equation*}
	Moreover,
	\begin{equation*}
		\FO\hookrightarrow\TC
	\end{equation*}
	with 
	\begin{equation*}
		\tr(S)=\intrd\K{S}(x,x)\,dx.
	\end{equation*}
\end{lemma}
\begin{proof}
	We just have to prove the continuous inclusion of Feichtinger operators into $\TC$, all the remaining statements can be found in \cite{FeiJak2022}, see in particular Corollary 3.15 and Remark 9. The claim follows from an elementary computation:
	\begin{align*}
		\norm{S}_{\TC}&=\abs{\tr(A)}\leq\intrd\sumn\abs{f_n(x)g_n(x)}\,dx=\sumn\intrd\abs{f_n(x)g_n(x)}\,dx\\
		&\leq\sumn\norm{f_n}_{L^2}\norm{g_n}_{L^2}\lesssim\sumn\normFA{f_n}\normFA{g_n}<\infty.
	\end{align*}
	Since $\FA(\rdd)=\FA(\rd)\hat{\otimes}\FA(\rd)$, see e.g. \cite[Lemma 2.1]{FeiJak2022}, we get
	\begin{equation*}
		\norm{S}_{\TC}\lesssim\normFA{\K{S}}\asymp\normFO{S},
	\end{equation*}
which gives the desired assertion.
\end{proof}
The preceding result and the observations in \cite[p. 4]{FeiJak2022} yield 
\begin{equation}
	\FO \hookrightarrow\TC\hookrightarrow\HS\hookrightarrow B(L^2(\rd))\hookrightarrow\MO.
\end{equation}
The fact that all Feichtinger operators are trace class implies the validity of Lemma \ref{Lem-QHA-new-tools}.

\subsection{$\tau$-quantization of operators}
The following remark is the key insight for the subsequent results concerning $\Opt$ and $\tW$.
\begin{remark}\label{Rem-key-observation}
	Let us consider $f,g\in L^2(\rd)$ such that $f\neq 0$, $a\in L^2(\rdd)$ and $\seqj{f}$ o.n.b. for $L^2$ with $f_1=f$. Then we compute as follows:
	\begin{align*}
		\la \Opt(a)f,g\ra &= \la \Opt(a)f,\sumj \la g,f_j\ra f_j\ra=\sumj \la \Opt(a)\left(\la f_j,g\ra f\right),f_j\ra\\
		&=\sumj\la \Opt(a)(f\otimes g) f_j, f_j\ra=\tr\left(\Opt(a)(f\otimes g)\right).
	\end{align*}
	Taking into account the weak definition of $\Opt(a)$ and \eqref{Eq-Def-tauWigner-operators} we can write
	\begin{equation}
		\la \Opt(a)f,g\ra=\la a, \tW((f\otimes g)^\ast)\ra=\tr\left(\Opt(a)(f\otimes g)\right)=\la \Opt(a),( f\otimes g)^\ast\ra_{(\TC,\BO)}.
	\end{equation}
	By computations similar to the ones above for $S\in\TC$ with the spectral decomposition $\sumk \lambda_k f_k\otimes g_k$ after extending $\{f_k\}_k$ to an orthonormal basis of $L^2(\mathbb{R}^d)$ implies
	\begin{equation}\label{Eq-Rem-key-observation}
		\la a, \tW S\ra= \tr\left(\Opt(a) S^\ast \right)=\la \Opt(a), S\ra_{(\TC,\BO)}.
	\end{equation}
\end{remark}

\begin{theorem}\label{Th-Opt-tW-TC-BO}
		For every $\tau\in[0,1]$ the following mappings are linear and continuous:
	\begin{equation*}
		\Opt\colon L^2(\rdd)\to\BO,\qquad\tW\colon\TC\to L^2(\rdd).
	\end{equation*}
	Moreover, $\Opt$ is the Banach space adjoint of $\tW$: $\Opt=\tW^\ast$.
\end{theorem}
\begin{proof}
	The boundedness of $\Opt$ is evident; the proof of the continuity of $\tW$ follows by a similar reasoning as the proof of the subsequent Theorem \ref{Th-Opt-tW-FO-MD}. The last claim is just \eqref{Eq-Rem-key-observation}.
\end{proof}

\begin{theorem}\label{Th-Opt-tW-FO-MD}
	For every $\tau\in[0,1]$ the following mappings are linear and continuous:
	\begin{equation*}
		\Opt\colon \MD(\rdd)\to\MO,\qquad\tW\colon\FO\to \FA(\rdd).
	\end{equation*}
		Moreover, $\Opt$ is the Banach space adjoint of $\tW$: $\Opt=\tW^\ast$, i.e. for every $a\in\MD(\rdd)$ and $S\in\FO$
		\begin{equation}\label{Eq-Opt-adjoint-tW}
			\dual{a}{\tW S}=\dualOP{\Opt(a)}{S}.
		\end{equation}
\end{theorem}
\begin{proof}
	The boundedness and linearity of $\Opt$ follow from the definitions. By using the formal representation of $\Opt(a)$ we can derive an expression for its kernel:
	\begin{equation}\label{Eq-kernel-Opt-a}
		\K{\Opt(a)}(t,x)=\intrd e^{2\pi i (t-x)\om}a((1-\tau)t+\tau x,\om)\,d\om.
	\end{equation} 
	Let us consider first $f,g\in\FA$. Then a standard argument, see e.g. \cite[Proposition 1.3.25]{CordRod2020}, gives that
	\begin{equation*}
		\tW(f\otimes g)=\tW(f,g)\in\FA(\rdd)\qquad\text{with}\qquad\normFA{\tW(f\otimes g)}\lesssim\normFA{f}\normFA{g}.
	\end{equation*}
	Since Lemma \ref{Lem-QHA-new-tools} holds for $\FO$, we write $\tW=\cF_\sigma\FtW$ and use the spectral  decomposition for $S$ of the form $\sumn f_n\otimes g_n$ as shown in Lemma \ref{Lem-representation-FO}. Now, we compute:
	\begin{align}
		\FtW S(z)&=\tr(\tTFS(z)^\ast S)=\tr(\sumn \tTFS(z)^\ast(f_n\otimes g_n))\notag\\
		&=\sumn \la\tTFS(z)^\ast f_n,g_n\ra=\sumn \tV_{g_n}f_n(z). \label{Eq-FtW-effect}
	\end{align}
	Taking a suitable window for the norm on $\FA(\rdd)$ \cite[Theorem 5.3]{Jakobsen2018} we have
	\begin{equation*}
		\normFA{\FtW S}\leq\sumn\normFA{\tV_{g_n} f_n}=\sumn\normFA{f_n}\normFA{g_n}<+\infty.
	\end{equation*}
	Consequently,
	\begin{align*}
		\normFA{\FtW S}&\leq\inf\{\sumn\normFA{f_n}\normFA{g_n}, S=\sumn f_n\otimes g_n\}\\
		&\leq\inf\{\sumn\normFA{f_n}\normFA{g_n}, \K{S}=\sumn \K{f_n\otimes g_n}\}\\
		&=\normFA{\K{S}}\asymp\normFO{S}.
	\end{align*}
	We proved the boundedness of $\FtW\colon\FO\to\FA(\rdd)$, the continuity of the symplectic Fourier transform $\cF_\sigma\colon\FA(\rdd)\to\FA(\rdd)$ is well-known, and thus the continuity of $\tW\colon\FO\to\FA(\rdd)$ follows. Concerning the last claim, we proceed as follows:
	\begin{align*}
		\dualOP{\Opt(a)}{S}&=\dual{\K{\Opt(a)}}{\K{S}}=\dual{\K{\Opt(a)}}{\sumn\K{f_\otimes g_n}}\\
		&=\sumn \dual{\K{\Opt(a)}}{\K{f_\otimes g_n}}=\sumn\dual{\Opt(a)g_n}{f_n}\\
		&=\sumn\dual{a}{\tW(f_n\otimes g_n)}=\dual{a}{\sumn\tW(f_n\otimes g_n)}\\
		&=\dual{a}{\tW S},
	\end{align*}
	which concludes the proof.
\end{proof}

On account of Theorem \ref{Th-Opt-tW-TC-BO} and \ref{Th-Opt-tW-FO-MD}, it seems reasonable to interpret $\tW S$ as the $\tau$-quantization of an operator in $\FO$ or $\TC$. 

\begin{corollary}\label{Cor-tW-Opt-iso-FO-FA}
	\begin{itemize}
		\item[$(i)$] For every $\tau\in[0,1]$ the mapping $\tW\colon\FO\to\FA(\rdd)$ is a topological isomorphism with inverse given by $\Opt\colon\FA(\rdd)\to\FO$;
		\item[$(ii)$] A linear and continuous operator $S\colon \FA(\rd)\to\MD(\rd)$ belongs to $\FO$ if and only if $\tW S\in \FA(\rdd)$ for some (and hence any) $\tau\in[0,1]$.
	\end{itemize}
\end{corollary}
\begin{proof}
	$(i)$ We observed in \eqref{Eq-Rem-tsym} that $\tW S$ is just the $\tau$-symbol $\tsym{S}$ of a trace class operator $S$, in particular this holds for $S\in\FO$. Therefore,  
	\begin{equation*}
		\Opt\circ\tW S=\Opt(\tsym{S})=S.
	\end{equation*}
	We now show that if we start with $a\in\FA(\rdd)$, then $\Opt(a)$ belongs to $\FO$. From \eqref{Eq-kernel-Opt-a}, we have that the kernel of $\Opt(a)$ can be written as
	\begin{equation*}
		\K{\Opt(a)}(t,x)=\intrd e^{2\pi i (t-x)\om}a((1-\tau)t+\tau x,\om)\,d\om=\Psi_\tau \cF^{-1}_2 a(t,x),
	\end{equation*}
	where $\cF^{-1}_2$ is the inverse of the partial Fourier transform with respect to the second variable; $\Psi_\tau$ is the change of variables induced by the matrix  
	\begin{equation}
		\begin{bmatrix}
			1-\tau&\tau\\
			1&-1
		\end{bmatrix},\qquad \Psi_\tau F(t,x)\coloneqq F((1-\tau)t+\tau x, t-x).
	\end{equation}
 	From the assumption $a$ in the Feichtinger algebra $\FA(\rdd)$ we have $\cF^{-1}_2a\in\FA(\rdd)$, thus $\Psi_\tau\cF^{-1}_2 a$ is in $\FA(\rdd)$, i.e. $\Opt(a)$ is an element of $\FO$. The fact that $\Opt$ is continuous from $\FA(\rdd)$ into $\FO$ is evident from the applications of $\cF^{-1}_2$ and $\Psi_\tau$. Hence we have shown that 
	\begin{equation*}
		\tW\circ\Opt(a)=\tsym{\Opt(a)}=a.
	\end{equation*}
	$(ii)$ The claim is a straightforward consequence of $(i)$.
\end{proof}

\begin{corollary}\label{Cor-FtW-tSR-iso-FO-FA}
	\begin{itemize}
		\item[$(i)$] 	For every $\tau\in[0,1]$ $\FtW\colon \FO\to \FA(\rdd)$ is a topological isomorphisms with inverse given by the $\tau$-spreading representation
		\begin{equation}\label{Eq-Def-tSR-FA}
			\tSR\colon\FA(\rdd)\to\FO\,,a\mapsto\intrdd a(z)\tTFS(z)\,dz;
		\end{equation}
		\item[$(ii)$]Let us define 
		\begin{equation}\label{Eq-Def-tSR-MD}
			\tSR\colon\MD(\rdd)\to\MO\,a\mapsto\intrdd a(z)\tTFS(z)\,dz,
		\end{equation}
		where the integral has to be understood weakly as follows:
		\begin{equation*}
			\dual{\tSR (a)f}{g}\coloneqq\dual{a}{\tV_{f}g},\qquad a\in\MD(\rdd),\,f,g\in\FA(\rd).
		\end{equation*}
		Then $\tSR$ as in \eqref{Eq-Def-tSR-MD} is well-defined, linear, continuous, extends  \eqref{Eq-Def-tSR-FA} and it is the Banach space adjoint of $\FtW$ in $(i)$:
		\begin{equation}
			\tSR=\FtW^\ast,
		\end{equation}
		in the sense that for every $a\in\MD(\rdd)$ and $S\in\FO$
		\begin{equation*}
			\dual{a}{\FtW S}=\dualOP{\tSR (a)}{S}=\dual{\K{\tSR(a)}}{\K{S}};
		\end{equation*}
		\item[$(iii)$] Every function $F\in\FA(\rdd)$ admits an expansion of the following type:
		\begin{equation*}
			F=\sumn \tV_{g_n}f_n,
		\end{equation*} 
		for some sequences $\seqn{f},\seqn{g}\subseteq\FA(\rd)$ such that $\sumn\normFA{f_n}\normFA{g_n}<\infty$.
	\end{itemize}
\end{corollary}
\begin{proof}
	$(i)$ First we notice that if we start with $a\in\FA(\rdd)$, then $\tSR (a)$ is the Feichtinger operator with kernel 
	\begin{equation*}
		\K{\tSR (a)}(y,u)=\intrd a(y-u,\om)e^{2\pi i y\om}\,d\om=\cF^{-1}_2[a(y-u,\cdot)](y).
	\end{equation*}
	Clearly $\tSR$ is continuous from $\FA(\rdd)$ into $\FO$.\\
	Since we have $\tW=\cF_{\sigma}\FtW$ and $\cF_{\sigma}$ is an automorphism of $\FA(\rdd)$, we can write $\FtW=\cF_{\sigma}\tW$ and which is an isomorphism due to Corollary \ref{Cor-tW-Opt-iso-FO-FA}.
	To prove that $\tSR$ is the inverse of $\FtW$ we use \eqref{Eq-FtW-effect}, take $S=\sumn f_n\otimes g_n\in\FO$ and $\psi,\f\in\FA(\rd)$:
	\begin{align*}
		\dual{(\tSR\circ\FtW S)\psi}{\f}&=\intrdd\FtW S(z)\dual{\tTFS(z)\psi}{\f}\,dz\\
		&=\sumn\intrdd \tV_{g_n}f_n(z)\overline{\tV_{\psi}\f(z)}\,dz\\
		&=\sumn\dual{f_n}{\f}\overline{\dual{g_n}{\psi}}\\
		&=\dual{\sumn\dual{\psi}{g_n}f_n}{\f}\\
		&=\dual{\sumn(f_n\otimes g_n)\psi}{\f}\\
		&=\dual{S\psi}{\f},
	\end{align*}
	in the third equality we used Moyal's identity. For the composition $\FtW\circ\tSR$, notice that this is the identity on $\FA(\rdd)$ due lo Lemma \ref{Lem-QHA-new-tools} $(vii)$.\\
	$(ii)$ Well-posedness, linearity and continuity of $\tSR$ from $\MD(\rdd)$ into $\MO$ are standard. Trivially \eqref{Eq-Def-tSR-MD} extends \eqref{Eq-Def-tSR-FA}. To see that $\tSR$ is the Banach space adjoint of $\FtW$ from $\FO$ into $\FA(\rdd)$, take $a\in\MD(\rdd)$ and $S\in\FO$. In the following calculations we use: the prior stated  \eqref{Eq-FtW-effect}, the representation for Feichtinger operators and their kernel given in Lemma \ref{Lem-representation-FO}, the Outer and Inner Kernel Theorems:
	\begin{align*}
		\dual{a}{\FtW S}&=\sumn\dual{a}{\tV_{g_n}f_n}=\sumn\dual{\tSR(a)g_n}{f_n}\\
		&=\sumn\dual{\K{\tSR(a)}}{\K{f_n\otimes g_n}}=\dual{\K{\tSR (a)}}{\K{S}}\\
		&=\dualOP{\tSR(a)}{S}.
	\end{align*}
	$(iii)$ The last  claim is a direct consequence of the computations in \eqref{Eq-FtW-effect} and the surjectivity of $\FtW$. 
\end{proof}

\subsection{A convenient environment for QHA} 

In Section \ref{Sec-Preliminaries} we introduced convolutions between a function and an operator and two operators. Keyl, Kiukas and Werner \cite{KeyKiuWer2016} showed that such convolutions make sense for wider classes of (generalized) functions and operators. We summarize here the main results; in what follows $\mathfrak{S}$ denotes the set of pseudo-differential operators with Weyl symbol in the Schwartz class $\Sch(\rdd)$ and $\mathfrak{S}'$ those pseudo-differential operators with Weyl symbol in $\Tem(\rdd)$. On account of the Schwartz Kernel Theorem we can identify $\mathfrak{S}'$ with the continuous and linear operators from $\Sch(\rd)$ into $\Tem(\rd)$.

\begin{proposition}
	\begin{itemize}
		\item[$(i)$] Suppose $S,T\in\mathfrak{S}$, $A\in\mathfrak{S}'$, $b\in\Sch(\rdd)$ and $a\in\Tem(\rdd)$. Then the following convolutions are well-defined and they extend the ones defined in Subsection \ref{Subsec-tools-QHA}:
		\begin{equation*}
			S\star T\in\Sch(\rdd), \quad S\star A\in\Tem(\rdd), \quad b\star S\in\mathfrak{S},\quad a\star S,b\star A\in\mathfrak{S}';
		\end{equation*}
		\item[$(ii)$]The Fourier-Wigner transform can be extended to a topological isomorphism $\FohW\colon\mathfrak{S}'\to\Tem(\rdd)$;
		\item[$(iii)$] We have $\cF_\sigma(S\star T)=\FohW S\cdot\FohW T$ and $\FohW(b\star S)=\cF_\sigma b\cdot \FohW S$ whenever $S,T$ and $b$ are such that the convolutions are defined as in part $(i)$;
		\item[$(iv)$] The Weyl symbol of $A\in\mathfrak{S}'$ is given by $\cF_\sigma\FohW A$.
	\end{itemize}
\end{proposition}

The authors of \cite{KeyKiuWer2016}  proved that the class of so-called Schwartz operators $\mathfrak{S}$ has the structure of a Fr\'echet space. We propose that the Banach space of Feichtinger operators $\FO$ is an alternative to $\mathfrak{S}$ that is a much bigger class of ``nice" operators. We start with  some preliminaries on $\FA$ and $\FO$.
\begin{lemma}\label{Lem-S0-sequentially-w-ast-dense}
	Given $f\in \MD(\rd)$, there exists a sequence $\netn{f}\subseteq \FA(\rd)$ which w-$\ast$ converges to $f$ and it is bounded by $\norm{f}_{\MD}$, i.e.
	\begin{equation*}
		\limnetn\la f_n,g	\ra = \dual{f}{g}\qquad\forall\,g\in \FA(\rd), \qquad\sup_{n}\norm{f_n}_{\FA}\leq\norm{f}_{\MD}.
	\end{equation*}
\end{lemma}
\begin{proof}
	Let us fix $f\in \MD(\rd)\smallsetminus \{0\}$ and set $R\coloneqq\norm{f}_{\MD}$. By \cite[Proposition 6.15]{Jakobsen2018}, there exists a net $\{f_\al\}_{\al\in A}\subseteq \FA(\rd)$ which converges w-$\ast$ to $f$ in $\MD$ and such that $\norm{f_\al}_{\MD}\leq R$ for every $\al\in A$. Set  
	\begin{equation*}
		B_R\coloneqq\left\{f\in \MD(\rd)\,|\,\norm{f}_{\MD}\leq R\right\}\quad\text{and}\quad E_R\coloneqq \FA(\rd)\cap B_R,
	\end{equation*}
	where $\FA$ is identified with its natural embedding in $\MD$, i.e. 
	\begin{equation*}
		E_R\subseteq B_R \subseteq \overline{E_R}^{w-\ast}.
	\end{equation*}
	$\overline{E_R}^{w-\ast}$ is bounded in $\MD(\rd)$. 
	\\
	In fact, if $f_0\in \overline{E_R}^{w-\ast}$, then there exists a net $\{f_\al\}_{\al\in A}\subseteq E_R$ that it  converges w-$\ast$ to $f_0$. Hence, we obtain
	\begin{equation*}
		\norm{f_0}_{\MD}\leq \liminf_{\al\in A}  \norm{f_\al}_{\MD}=\lim_{\al\in A}\inf\{\norm{f_\be}_{\MD}\,|\, \al\preceq \be\}\leq \lim_{\al\in A} R=R.
	\end{equation*}
	In particular, this shows that $\overline{E_R}^{w-\ast}\subseteq B_R$ and we get
	\begin{equation*}
		\overline{E_R}^{w-\ast}=B_R.
	\end{equation*}
	Since $\FA$ separable, and the relative w-$\ast$ topology on $B_R$
	is induced by a metric by \cite[Therem 2.6.23]{Megginson}. Hence the topological w-$\ast$ closure of $E_R$ equals its sequential w-$\ast$ closure. Consequently, there exists a sequence $\{f_n\}_{n}\subseteq E_R$ which converges w-$\ast$ to $f$ in $\MD(\rd)$.
\end{proof}

\begin{remark}
	The above lemma holds also for any LCA second countable group $\cG$ replacing $\rd$, see \cite[Theorem 2]{deVries1978} for the separability of $\FA(\cG)$.
\end{remark}

\begin{lemma}
	For any $S\in\MO$, there exists a sequence $\netn{S}\subseteq \FO$ such that 
	\begin{itemize}
		\item[$(i)$] $\normMO{S_n}\lesssim\normMO{S}$;
		\item[$(ii)$] $\limnetn\abs{\dual{(S-S_n)f}{g}}=0$ for all $f,g\in\FA(\rd)$.
	\end{itemize}
\end{lemma}
\begin{proof}
	This is a straightforward application of the Kernel Theorems \ref{Th-OuterKernel} and \ref{Th-OuterKernel} and of Lemma \ref{Lem-S0-sequentially-w-ast-dense}.
\end{proof}

Convergence as in item $(ii)$ of the above lemma will be also denoted by
\begin{equation*}
	\wsconvMOn{S}{S}\qquad\text{or}\qquad S=\wslimMOn S_n\qquad\text{in}\quad\MO.
\end{equation*}

\begin{lemma}\label{Lem-adjoint-FO}
	Let $S\colon \FA\to\MD$ be in $\FO$. Then the Banach space adjoint $S^\ast\colon\MD\to\FA$ is in $\FO$ with kernel 
	\begin{equation}\label{Eq-ker-adjoint}
		\K{S^\ast}(y,u)=\overline{\K{S}(u,y)}.
	\end{equation}
\end{lemma}
\begin{proof}
	We take $f,g\in\FA(\rd)$, then
	\begin{align*}
		\dual{Sf}{g}&=\intrdd\K{T}(y,u)\overline{g(y)}f(u)\,dydu\\
		&=\intrd f(u)\overline{\intrd\overline{\K{S}(y,u)}g(y)\,dy}\,du\\
		&=\dual{f}{S^\ast g}.
	\end{align*}
	Hence, $S^\ast g(y)=\intrd \overline{\K{S}(u,y)}g(u)\,du$, i.e. $\K{S^\ast}(y,u)=\overline{\K{S}(u,y)}$ which is an element of $\FA(\rdd)$.
\end{proof}

\begin{corollary}
	$\FO$ is a Banach $\ast$-algebra. 
\end{corollary}

We notice that $(S^\ast)\check{{}}=(\check{S})^\ast$, so that from now on we shall simply write $\check{S}^\ast$ when necessary.

\begin{lemma}\label{Lem-conv-FO-FO-FA-FO}
	\begin{itemize}
		\item[$(i)$] The following applications are surjective isometries:
		\begin{itemize}
			\item[$(i-a)$] $\alpha_z\colon\FO\to\FO$, for any $z=(x,\om)\in\rdd$, and 
			\begin{equation}\label{Eq-ker-alpha-z-S}
				\K{\alpha_z S}(y,u)=e^{2\pi i (y-u)\om}\K{S}(y-x,u-x);
			\end{equation}
			\item[$(i-b)$] $\check{\cdot}\colon\FO\to\FO$ and 
			\begin{equation}\label{Eq-ker-check-S}
				\K{\check{S}}(y,u)=\K{S}(-y-u);	
			\end{equation}
			\item[$(i-c)$] $\alpha_z\colon\MO\to\MO$, for any $z\in\rdd$;
			\item[$(i-d)$] $\check{\cdot}\colon\MO\to\MO$;
		\end{itemize} 
		\item[$(ii)$] Let $S,T\in\FO$ and $b\in\FA(\rdd)$. Then
		\begin{equation*}
			S\star T\in\FA(\rdd),\qquad b\star S\in\FO;
		\end{equation*}
		\item[$(iii)$] The kernel of the mixed-state localization operator $b\star S$ is given by
		\begin{equation}
			\K{b\star S}(y,u)=\intrd b(x,\om) e^{2\pi i (y-u)\om}\K{S}(y-x,u-x)\,dxd\om;
		\end{equation}
		for very $z=(x,\om)\in\rdd$ the kernel of $S\alpha_z\check{T}$ is
		\begin{equation}
			\K{S\alpha_z \check{T}}(y,u)=\intrd e^{2\pi i (y-t)\om}\K{T}(x-y,x-t)\K{S}(t,u)\,dt.
		\end{equation}
	\end{itemize}
\end{lemma}
\begin{proof}
	$(i)$ We leave the elementary computations to the interest reader, and note that in order to prove $\alpha_z S,\check{S}\in\FO$ the result \cite[Corollary 3.3]{FeiJak2022} is useful. A continuous and linear operator $S\colon\FA\to\MD$ is a Feichtinger operator if and only if
	\begin{equation*}
		\intrdd\intrdd \abs{\dual{S\pi(z)g_1}{\pi(w)g_2}}\,dzdw
	\end{equation*}
	is finite for any $g_1,g_2\in\FA(\rd)$.\\
	$(ii)$  We first address the convolution between two Feichtinger operators. By item $(i)$ and the fact that $\FO$ is a Banach algebra under composition, we have that $S\alpha_z \check{T}$ is in $\FO$ for any $z=(x,\om)\in\rdd$. We have by \cite[Corollary 3.15]{FeiJak2022}:
	\begin{align*}
		S\star T(z)&=\tr(S\alpha_z \check{T})=\intrd\K{S\alpha_z \check{T}}(y,y)\,dy=\intrdd\K{\alpha_z\check{T}}(y,t)\K{S}(t,y)\,dtdy\\
		&=\intrdd e^{2\pi i (y-t)\om}\K{T}(x-y,x-t)\K{S}(t,y)\,dtdy\\
		&=\intrd\left(\intrd \K{T}(x-y,x-t)\K{S}(t,y)e^{-2\pi i t\om}\,dt\right)e^{2\pi i y\om}\,dy\\
		&=\cF^{-1}_2\cF_1\left(\Phi T_{(x,x)}\K{T}\cdot \K{S}\right)(\om,\om),
	\end{align*}
	where $\Phi F(t,y)\coloneqq F(-y,-t)$, $\cF_1$ and $\cF_2$ are the partial Fourier transforms with respect to the first and second variable, respectively. Consider now $f,g,h,l\in\FA(\rd)$, it is useful to compute the following where $P$ is the parity operator:
	\begin{align*}
		\cF^{-1}_2\cF_1\left(\Phi T_{(x,x)}\K{h\otimes l}\cdot\K{f\otimes g}\right)&(\om,\om)=\intrd\left(\intrd h(x-y)\overline{l(x-t)}f(t)\overline{g(y)}e^{-2\pi i t\om}\,dt\right)e^{2\pi i y\om}\,dy\\
		&=\intrd f(t)e^{-2\pi i t\om}\overline{l(x-t)}\,dt\cdot\intrd \overline{g(y)}e^{2\pi i y\om}h(x-y)\,dy\\
		&=V_{P l}f(-x,\om)\cdot\overline{V_{P h}g(-x,\om)}.
	\end{align*}
	Hence $\cF^{-1}_2\cF_1\left(\Phi T_{(x,x)}\K{h\otimes l}\cdot\K{f\otimes g}\right)(\om,\om)$ is in $\FA(\rdd)$ as a function of $(x,\om)$. We consider now two representations $S=\sumn f_n\otimes g_n$ and $T=\sumn h_n\otimes l_n$, see Lemma \ref{Lem-representation-FO}, so that 
	\begin{equation*}
		\K{S}=\sumn \K{f_n\otimes g_n},\qquad\K{T}=\sumn\K{h_n\otimes l_n}.
	\end{equation*}
	It follows that we can write
	\begin{align*}
		S\star T(z)&=\cF^{-1}_2\cF_1\left(\Phi T_{(x,x)}\sum_M^\infty\K{h_m\otimes l_m}\cdot \sumn \K{f_n\otimes g_n} \right)(\om,\om)\\
		&=\sum_{m=1}^\infty\sumn \cF^{-1}_2\cF_1\left(\Phi T_{(x,x)}\K{h_m\otimes l_m}\cdot\K{f_n\otimes g_n}\right)(\om,\om)\\
		&=\sum_{m=1}^\infty\sumn V_{P l_m}f_n(-x,\om)\cdot\overline{V_{P h_m}g_n(-x,\om)}\in\FA(\rdd),
	\end{align*}
	the convergence is guaranteed by Lemma \ref{Lem-representation-FO}.\\
	Concerning $b\star S$, the following estimate for any $f,g\in\FA(\rd)$ proves that $b\star S\in\MO$:
	\begin{equation*}
		\abs{\dual{(b\star S)f}{g}}\leq \intrdd \abs{b(z)}\abs{\dual{S\pi(z)^\ast f}{\pi(z)^\ast g}}\,dz\lesssim \norm{b}_{L^1}\normMO{S}\normFA{f}\normFA{g}.
	\end{equation*}
	We exploit \cite[Theorem 3.2 (ii)]{FeiJak2022} to show that $b\star S$ is in $\FO$. For $g_1,g_2\in\FA(\rd)$ we have 
	\begin{align*}
		\intrdd\intrdd& \abs{\dual{(b\star S)\pi(w)g_1}{\pi(u)g_2}}\,dwdu\leq\intrdd\intrdd\intrdd\abs{b(z)}\\
		&\times\abs{\dual{S\pi(w-z)g_1}{\pi(u-z)g_2}}\,dzdwdu\\
		&=\intrdd\intrdd\abs{\dual{S\pi(w')g_1}{\pi(u')g_2}}\,dw'du'\cdot\intrdd\abs{b(z)}\,dz<+\infty.
	\end{align*}
	$(iii)$ We compute explicitly the kernel of the operator given by the convolution $b\star S$:
	\begin{align*}
		\dual{(b\star S)f}{g}&=\intrdd b(x,\om)\intrdd \K{S}(y,u)\overline{\pi(-z)g(y)}\pi(-z)f(u)\,dydu\,dxd\om\\
		&=\intrdd \intrdd b(x,\om)e^{2\pi i (y-u)\om}\K{S}(y,u)\overline{g(y+x)}f(u+x)\,dxd\om\,dydu,
	\end{align*}
	for $z=(x,\om)\in\rdd$. The change of variables $y'=y+u, u'=u+x$ gives the desired result. The last claim is just a direct application of \eqref{Eq-ker-alpha-z-S}, \eqref{Eq-ker-check-S} and the Banach algebra property for $\FO$ \cite[Lemma 3.10]{FeiJak2022}.
\end{proof}

\begin{corollary}
	Let $S,T\in\FO$ with spectral decompositions $S=\sumn f_n\otimes g_n$ and $T=\sumn h_n\otimes l_n$, where $ \seqn{f},\seqn{g},\seqn{h},\seqn{l}\subseteq
	\FA(\rd)$ with $\sumn\normFA{f_n}\normFA{g_n}<+\infty$, $\sumn\normFA{h_n}\normFA{l_n}<+\infty$.
	Then, with the notations introduced in the proof of Lemma \ref{Lem-conv-FO-FO-FA-FO},  for every $z=(x,\om)\in\rdd$:
	\begin{align}
		S\star T(z)&=\cF^{-1}_2\cF_1\left(\Phi T_{(x,x)}\K{T}\cdot \K{S}\right)(\om,\om)\notag\\
		&=\sum_{m=1}^\infty\sumn V_{P l_m}f_n(-x,\om)\cdot\overline{V_{P h_m}g_n(-x,\om)}.
	\end{align}
\end{corollary}

\begin{definition}\label{Def-conv-FO-MO-MD-FO-FA-MO}
	Let $A\in\MO$, $a\in\MD(\rdd)$, $S\in\FO$ and $b\in\FA(\rdd)$. Consider any sequences $\netn{A}\subseteq \FO$ and $\netn{a}\subseteq\FA(\rdd)$ such that 
	\begin{equation*}
		\wsconvMOn{A}{A}\qquad\text{and}\qquad\wsconvMDn{a}{a}.
	\end{equation*} 
	Then we define:
	\begin{align}
		S\star A&\coloneqq\wslimMOn S\star A_n\qquad\text{in}\quad\MD(\rdd);\label{Eq-Def-S-star-A}\\
		a\star S\coloneqq S\star a&\coloneqq\wslimMOn a_n\star S\qquad\,\text{in}\quad\MO;\label{Eq-Def-a-star-S}\\
		b\star A\coloneqq A\star b&\coloneqq\wslimMOn b\star A_n\qquad\,\text{in}\quad\MO.\label{Eq-Def-b-star-A}
	\end{align}
\end{definition}

\begin{remark}\label{Rem-useful-identities}
	The reader may find it useful to keep in mind the following simple identities, which will be used in the proof of the subsequent proposition. Consider $S\in\FO, \psi,\f,f,g\in\FA(\rd)$ and $z\in\rdd$:
	\begin{align*}
		\alpha_z(\psi\otimes\f)&=\pi(z)\psi\otimes\pi(z)\f;\\
		(\psi\otimes\f)(\K{f\otimes\overline{g}})&=\la f,\f\ra (\psi\otimes g);\\
		(\psi\otimes\f)\star \check{S}(z)&
		=\dual{\pi(z)S\pi(z)^\ast \psi}{\f}.
	\end{align*}
\end{remark}

\begin{proposition}
	The convolutions introduced in Definition \ref{Def-conv-FO-MO-MD-FO-FA-MO}: 
	\begin{itemize}
		\item[$(i)$] They do not depend on the sequences chosen; moreover, taking $A,a,S,b$ as in Definition \ref{Def-conv-FO-MO-MD-FO-FA-MO}:
		\begin{align}
			\dual{S\star A}{b}&=\dual{\K{A}}{\K{b\star \check{S}^\ast}};\label{Eq-action-S-star-A}\\
			\dual{(a\star S)f}{g}&=\dual{a}{(g\otimes f)\star\check{S}^\ast};\label{Eq-action-a-star-S}\\
			\dual{(b\star A)f}{g}&=\dual{\K{A}}{\K{b^\ast \star(g\otimes f)}},\label{Eq-action-b-star-A}
		\end{align}
		where $b^\ast(z)\coloneqq \overline{b(-z)}$;
		\item[$(ii)$] These extend the definitions given in Subsection \ref{Subsec-tools-QHA};
		\item[$(iii)$] They are commutative; 
		\item[$(iv)$] Moreover, they are associative. In particular,  if $z\in\rdd$, $T,Q\in\FO$, $\sigma\in\FA(\rdd)$ and $A,a,S,b$ as in Definition \ref{Def-conv-FO-MO-MD-FO-FA-MO} then:
		\begin{align}
			(S\star(T\star b))(z)&=((S\star T)\ast b)(z);\label{Eq-assoc-S-T-b}\\
			S\star(T\star Q)&=(S\star T)\star Q;\label{Eq-assoc-S-T-Q}\\
			(S\star b)\star \sigma&=S\star (b\ast \sigma);\label{Eq-assoc-S-b-sigma}\\
			S\star (T\star a)&=(S\star T)\ast a;\\
			A\star (T\star b)&=(A\star T)\star b;\\
			S\star(T\star A)&=(S\star T)\star A;	
		\end{align}
		in the above identities $\ast$ denotes the usual convolution between two functions or a function and a distribution.
	\end{itemize}
\end{proposition}
\begin{proof}
	$(i)$ It suffices to show \eqref{Eq-action-S-star-A}, \eqref{Eq-action-a-star-S} and \eqref{Eq-action-b-star-A}, since the other assertions in $(i)$ are evident.\\
	 We start with\eqref{Eq-action-S-star-A}. Let $b\in\FA(\rdd)$ and $z=(x,\om)\in\rdd$, in the subsequent computations we use Lemma \ref{Lem-adjoint-FO} and \ref{Lem-conv-FO-FO-FA-FO}:
	\begin{align*}
		\dual{S\star A}{b}&=\limnetn\dual{S\star A_n}{b}=\limnetn\intrdd\tr(S\al_z\check{A}_n)\overline{b(z)}\,dz\\
		&=\limnetn\intrdd\intrd \K{S\al_z\check{A}_n}(y,y)\,dy \overline{b(z)}\,dz\\
		&=\limnetn\intrdd\intrd \intrd e^{2\pi i (y-t)\om}\K{A_n}(x-y,x-t)\K{S}(t,y)\,dtdy\,\overline{b(z)}\,dz\\
		&=\limnetn\intrdd\intrd\intrd e^{2\pi i (t'-y')\om}\K{A_n}(y',t')\K{S}(x-t',x-y')\,dt'dy'\,\overline{b(z)}\,dz\\
		&=\limnetn\intrd\intrd\K{A_n}(y',t')\overline{\left(\intrdd \overline{\K{S}(x-t',x-y')}e^{2\pi i (y'-t')\om}b(z)\,dz\right)}\,dy'dt'\\
		&=\limnetn\intrd\intrd\K{A_n}(y',t')\overline{\left(\intrdd \overline{\K{\check{S}}(t'-x,y'-x)}e^{2\pi i (y'-t')\om}b(z)\,dz\right)}\,dy'dt'\\
		&=\limnetn\intrd\intrd\K{A_n}(y',t')\overline{\left(\intrdd \K{\check{S}^\ast}(y'-x,t'-x)e^{2\pi i (y'-t')\om}b(z)\,dz\right)}\,dy'dt'\\
		&=\limnetn\intrd\intrd\K{A_n}(y',t')\overline{\K{b\star \check{S}^\ast}(y',t')}\,dy'dt'.
	\end{align*}
	About \eqref{Eq-action-a-star-S}, we take $f,g\in\FA(\rd)$ and compute directly keeping in mind Remark \ref{Rem-useful-identities}:
	\begin{align*}
		\dual{(a\star S)f}{g}&=\limnetn\intrdd a_n(z)\dual{\pi(z)S\pi(z)^\ast f}{g}\,dz\\
		&=\limnetn\intrdd a_n(z)\overline{\dual{\pi(z)S^\ast\pi(z)^\ast g}{f}}\,dz\\
		&=\limnetn\intrdd a_n(z)\overline{(g\otimes f)\star \check{S}^\ast (z)}\,dz.
	\end{align*}
	Let us address \eqref{Eq-action-b-star-A}:
	\begin{align*}
		\dual{(b\star A)f}{g}&=\limnetn\dual{\K{b\star A_n}}{\K{g\otimes f}}\\
		&=\limnetn\intrdd \Big(\intrdd b(x,\om) e^{2\pi i (y-u)\om}\K{A_n}(y-x,u-x)\,dxd\om\Big)\\
		&\times\overline{g(y)}f(u)\,dydu\\
		&=\limnetn\intrdd \K{A_n}(y',u') \overline{\Big(\intrdd \overline{b(x,\om)}e^{-2\pi i (y'-u')\om}}\\
		&\times\overline{g(y'+x)\overline{f(u'+x)}\,dxd\om\Big)}\, dy'du'\\
		&=\limnetn\intrdd \K{A_n}(y',u')  \overline{\Big(\intrdd b^\ast(x',\om')e^{2\pi i (y'-u')\om'}}\\
		&\times\overline{g(y'-x')\overline{f(u'+x')}\,dx'd\om'\Big)}\,dy'du'\\
		&=\limnetn\intrdd  \K{A_n}(y',u') \overline{\K{b^\ast\star (g\otimes f)}(y',u')}\,dy'du',
	\end{align*}
	where for sake of brevity we set $b^\ast(z)\coloneqq \overline{b(-z)}$.\\
	$(ii)$ and $(iii)$ are trivial.\\
	$(iv)$ We prove just \eqref{Eq-assoc-S-T-b}, \eqref{Eq-assoc-S-T-Q} and \eqref{Eq-assoc-S-b-sigma}. The remaining identities can be derived in a similar manner.\\
	In order to show \eqref{Eq-assoc-S-T-b} we compute for $z\in\rdd$:
	\begin{align*}
		(S\star(T\star b))(z)&=\tr\left(S\circ \alpha_z\left(\left(\intrdd b(z)\alpha_w T\,dw\right)\check{{}}\right)\right)\\
		&=\tr\left(S\circ \left(\intrdd b(w)\alpha_z\left(\left(\alpha_w T\right)\check{{}}\right)\,dw\right)\right)\\
		&=\tr\left(S\circ\intrdd b(w)\alpha_z\alpha_{-w}\check{T}\,dw\right)\\
		&=\tr\left(S\circ\intrdd b(-w')\alpha_{w'}\alpha_{z}\check{T}\,dw'\right)\\
		&=\intrdd b(-w')\tr\left(S\alpha_{w'+z}\check{T}\right)\,dw',
	\end{align*}
	where the last equality is due, e.g., to \cite[Proposition 2.9]{Eirik-Master}. we can the rephrase the last right-side term as 
	\begin{align*}
		\intrdd b(z-w'')\tr\left(S\alpha_{w''}\check{T}\right)\,dw''&=\intrdd b(z-w'')(S\star T)(w'')\,dw''\\
		&=((S\star T)\ast b)(z).
	\end{align*}
	For the proof of \eqref{Eq-assoc-S-T-Q}, the following property of the trace is useful:
	\begin{equation*}
		\intrdd\tr(S\alpha_w T)\,dw=\tr(S)\tr(T),
	\end{equation*}
	where $S,T\in\TC$. Take now $f,g\in\FA(\rd)$:
	\begin{align*}
		\dual{(S\star(T\star Q))f}{g}&=\intrdd \tr(T\alpha_z\check{Q})\dual{\alpha_z S f}{g}\,dz\\
		&=\intrdd \tr(Q \alpha_z\check{T})\tr((\alpha_z S) (f\otimes g))\,dz\\
		&=\intrdd \intrdd \tr(Q (\alpha_z\check{T})\alpha_w((\alpha_z S)(f\otimes g)))\,dwdz\\
		&=\intrdd \intrdd \tr((f\otimes g)(\alpha_w Q)\alpha_z((\alpha_w \check{T})S))\,dzdw\\
		&=\intrdd \tr(S\alpha_w\check{T})\tr((\alpha_w Q)(f\otimes g))\,dw\\
		&=\dual{((S\star T)\star Q)f}{g}.
	\end{align*}
	Also the last identity \eqref{Eq-assoc-S-b-sigma} may be deduced by a direct computation. For $f,g\in\FA(\rd)$ we have
	\begin{align*}
		\dual{((S\star b)\star \sigma)f}{g}&=\intrdd \sigma(z)\dual{\alpha_z(S\star b)f}{g}\,dz\\
		&=\intrdd\sigma(z)\intrdd b(w)\dual{(\alpha_w S)\pi(z)^\ast f}{\pi(z)^\ast g}\,dwdz\\
		&=\intrdd\intrdd\sigma(z)b(w)\dual{(\alpha_{w+z}S)f}{g}\,dwdz\\
		&=\intrdd\intrdd\sigma(z)b(w)\tr((\alpha_{w+z}S)(f\otimes g))\,dwdz\\
		&=\intrdd b(w)\intrdd \sigma(z'-w)\tr((\alpha_{z'}S)(f\otimes g))\,dz'dw\\
		&=\intrdd (\intrdd b(w)\sigma(z'-w)\,dz')\tr((\alpha_{z'}S)(f\otimes g))\,dw\\
		&=\intrdd b\ast \sigma(z')\dual{(\alpha_{z'}S)f}{g}\,dz'\\
		&=\dual{(S\star (b\ast \sigma))f}{g}.
	\end{align*}
	This concludes the proof.
\end{proof}

\begin{corollary}
	The mappings $\FtW$ and $\tW$ defined on $\FO$ can be extended to topological isomorphisms
	\begin{equation*}
		\FtW\colon\MO\to\MD(\rdd)\qquad\text{and}\qquad\tW\colon\MO\to\MD(\rdd)
	\end{equation*} 
	by duality:
	\begin{equation}\label{Eq-Def-FtW-tW-MO}
		\dual{\FtW S}{a}\coloneqq\dualOP{S}{\tSR a},\qquad\dual{\tW S}{a}\coloneqq\dualOP{S}{\Opt a},
	\end{equation}
	where $S\in\MO$ and $a\in\FA(\rdd)$. The inverses are given by
	\begin{equation*}
		\tSR\colon\MD(\rdd)\to\MO\qquad\text{and}\qquad\Opt\colon\MD(\rdd)\to\MO,
	\end{equation*} 
	respectively.
\end{corollary}
\begin{proof}
	The definitions in \eqref{Eq-Def-FtW-tW-MO} rely on the fact that $\Opt=\tW^\ast$ and $\tSR=\FtW^\ast$, see Theorem \ref{Th-Opt-tW-FO-MD} and Corollary \ref{Cor-FtW-tSR-iso-FO-FA}. It is straightforward to see that if $S\in\MO$, then $\FtW S$ and $\tW S$ defined as in \eqref{Eq-Def-FtW-tW-MO} are in $\MD(\rdd)$. Also linearity and boundedness of $\FtW\colon\MO\to\MD(\rdd)$ and $\tW\colon\MO\to\MD(\rdd)$ are easy to verify as well as the fact that they extend  $\FtW\colon\FO\to\FA(\rdd)$ and $\tW\colon\FO\to\FA(\rdd)$.\\
	We show that $\tW$ is an isomorphisms with inverse $\Opt$, then $\FtW$ is treated in the same way. $\tW$ is injective because $\Opt\colon\FA(\rdd)\to\FO$ is an isomorphism. Fix now $a\in\MD(\rdd)$, there exists a sequence $\netn{a}\subseteq\FA(\rdd)$ such that $\wsconvMDn{a}{a}$. Since $\tW$ is an isomorphism between $\FO$ and $\FA(\rdd)$, there exists $\netn{A}\subseteq\FO$ such that $a_n=\tW A_n$. We see that there is $A\in\MO$ such that $\wsconvMOn{A}{A}$, in fact taking $b\in\FA(\rdd)$
	\begin{equation*}
		\dual{a}{b}=\limnetn\dual{\tW A_n}{b}=\limnetn\dualOP{A_n}{\Opt b}.
	\end{equation*}
	Hence $a=\tW A$, which proves that $\tW$ is onto. We show now that $\tW\circ\Opt$ is the identity on $\MD(\rdd)$, take $a\in\MD(\rdd)$ and $b\in\FA(\rdd)$:
	\begin{equation*}
		\dual{\tW\circ\Opt a}{b}=\dualOP{\Opt a}{\Opt b}=\dual{a}{\tW\circ\Opt b}=\dual{a}{b}.
	\end{equation*}
	The first identity is just \eqref{Eq-Def-FtW-tW-MO}, the second one is \eqref{Eq-Opt-adjoint-tW} and the last one is $(i)$ of Corollary \ref{Cor-tW-Opt-iso-FO-FA}. For the other direction, take $S\in\MO$ and $T\in\FO$:
	\begin{equation*}
		\dualOP{\Opt\circ\tW S}{T}=\dual{\tW S}{\tW T}=\dualOP{S}{\Opt\circ\tW T}=\dualOP{S}{T}.
	\end{equation*}
	The first identity is \eqref{Eq-Opt-adjoint-tW}, the second one is \eqref{Eq-Def-FtW-tW-MO} and the last one is $(i)$ of Corollary \ref{Cor-tW-Opt-iso-FO-FA}.
\end{proof}

	\subsection{$\tau$-Cohen's class of operators}\label{Ch-7-Subsec-tau-Cohen-class}
In the present subsection we define $\tQ{a}{S}$ and recall the definition of $\tQ{S}{f}$ from \cite{Luef1}. We shall see that $\tQ{a}{S}$ relates to well-known objects and observe that it coincides with the $\tau$-symbol of the mixed-state localization operator $a\star S$. We continue with some statements concerning the interplay between the Gabor matrix of an operator $\GM{T}$, the $\tau$-Cohen's class, the trace and the $\tau$-Wigner distribution.
\begin{definition}
	For $a\in\MD(\rdd)$ we define the $\tau$-Cohen's class distribution, with kernel $a$, of an operator $S\in\FO$ as 
	\begin{equation}\label{Eq-Def-Cohen-class-S-wrt-function}
		\tQ{a}{S}\coloneqq a\ast \tW S.
	\end{equation} 
\end{definition}

Of course, the rank-one case $f\otimes g$ reduces to the definition given in \eqref{Eq-Def-tau-Cohen-functions}.
We recall also the definition given in \cite{Luef1} of Cohen's class distribution of a function $f\in\FA(\rd)$ w.r.t. the operator $S\in\MO$ by
\begin{equation}\label{Eq-Def-Cohen-class-f-wrt-operator}
	Q_S f\coloneqq (f\otimes f)\star \check{S}.
\end{equation}
It can be easily seen that for every $z\in\rdd$
\begin{equation*}
	Q_S f(z)= (f\otimes f)\star \check{S}(z)=\la (\alpha_z S)f,f,\ra.
\end{equation*}

\begin{remark}
	If $a\in\MD(\rdd)$ and $S\in\FO$, then we see that the $\tau$-Cohen's class representation of $S$ w.r.t. $a$ is just the $\tau$-symbol of the mixed-state localization operator $a\star S$:
	\begin{equation*}
		\tsym{a\star S}=\tW(a\star S)=a\ast \tW S=\tQ{a}{S}.
	\end{equation*}
\end{remark}

\begin{lemma}
	Let $S\in\FO$ have the spectral decomposition $\sumn f_n\otimes g_n$, for $f,\f,\psi\in\FA(\rd)$ and $\seqn{h}\subseteq\FA(\rd)$ with
	\begin{equation*}
		\sumn\normFA{h_n}^2<+\infty.
	\end{equation*}. Then for every $z\in\rdd$:
	\begin{align}
		\tQ{\omtW(\check{\psi},\check{\f})}{S}(z)&=\sumn V_{\f}f_n(z)\overline{V_{\psi}g_n(z)};\\
		\tQ{\omtW(\check{\f},\check{\f})}{\sumn h_n\otimes h_n}(z)&= \sumn\abs{V_{\f}h_n(z)}^2.
	\end{align}
\end{lemma}
\begin{proof}
	Clearly, it suffices to prove the first identity. We show first that for $f,g\in\FA(\rd)$
	\begin{equation}
		\tQ{a}{f,g}=(f\otimes g)\star\Opomt(a).
	\end{equation}
	In fact, applying $\cF_{\sigma}$ to the right-hand side first we get
	\begin{equation*}
		\cF_{\sigma}((f\otimes g)\star\Opomt(a))=\FtW(f\otimes g)\cdot\FomtW\Opomt(a)=\tV_g f\cdot\cF_{\sigma}a.
	\end{equation*}
	We apply $\cF_{\sigma}$ a second time:
	\begin{equation*}
		(f\otimes g)\star\Opomt(a)=\cF_{\sigma}\tV_g f\ast\cF_{\sigma}\cF_{\sigma}a=\tW(f,g)\ast a.
	\end{equation*}
	We can now proceed as follows:
	\begin{align*}
		\tQ{\omtW(\check{\psi},\check{\f})}{S}&=\omtW(\check{\psi},\check{\f})\ast \tW(\sumn f_n\otimes g_n)=\sumn \omtW(\check{\psi},\check{\f})\ast\tW(f_n,g_n)\\
		&=\sumn (f_n\otimes g_n)\star \Opomt(\omtW(\check{\psi},\check{\f}))=\sumn(f_n\otimes g_n)\star(\check{\psi}\otimes\check{\f})\\
		&=\sumn V_{\f}f_n(z)\overline{V_{\psi}g_n(z)},
	\end{align*}
	where the last equality is due to \cite{Luef1}.
\end{proof}

We call a bounded operator $T$ on $L^2(\rd)$ positive, denoted by $T\geq 0$, if
\begin{equation*}
	\la Tf ,f\ra\geq 0,\qquad\forall\,f\in L^2(\rd).
\end{equation*}
An operator $T\in\TC$ and  $T\geq0$ is also called a state in quantum mechanics.\\
Let us take $T\in\MO$ and $\f\in\Sch$, then the Gabor matrix of $T$ (w.r.t. $\f$) is defined as
\begin{equation}
	\GM{T}(z,w)\coloneqq \la T\pi(w)\f,\pi(z)\f\ra,\qquad z=(x,\om),w=(u,v)\in\rdd.
\end{equation}
We notice that the Gabor matrix of an operator does not depend on $\tau$, in the sense that 
\begin{equation*}
	\GM{T}(z,w)= \la T\pi(w)\f,\pi(z)\f\ra= \la T\tTFS(w)\f,\tTFS(z)\f\ra,\qquad\forall\,\tau\in[0,1].
\end{equation*}

\begin{remark}\label{Rem-GaborMatrix-CohensClass}
	We point out that the diagonal of the Gabor matrix of $T$, w.r.t. $\f$, is the Cohen's class representation of $\f$ w.r.t. $T$ up to a reflection:
	\begin{equation}
		\GM{T}(-z,-z)=Q_T\f(z).
	\end{equation}
	In fact
	\begin{align*}
		\GM{T}(-z,-z)&=\la T\pi(-z)\f,\pi(-z)\f\ra=\la T\pi(z)^\ast\f,\pi(z)^\ast\f\ra\\
		&=\la (\alpha_z T)\f,\f,\ra=Q_T \f(z).
	\end{align*}
\end{remark}

Let $F$ and $H$ be functions of $(z,w)\in\bR^{4d}$ and let $\twistmatrix$ be a real $4d\times 4d$ matrix. Then the twisted convolution induced by $\twistmatrix$ is defined as 
\begin{equation}
	F\twistconv H(z,w)\coloneqq\intrdd \intrdd F(z',w')H(z-z',w-w')e^{2\pi i (z,w)\twistmatrix(z',w')}\,dz'dw'.
\end{equation}
%
%

\begin{lemma}\label{Lem-H.R.-Lem-2.7}
	Let $T,S\in \TC$, $T,S\geq 0$. Then for every $\tau\in[0,1]$ we have
	\begin{equation}
		\tr(TS)=\intrdd \tW T(z)\overline{\tW S(z)}\,dz.
	\end{equation}
\end{lemma}
\begin{proof}
	Since $T$ and $S$ are trace-class and positive, they can be described as 
	\begin{equation*}
		T=\sumn\lambda_n f_n\otimes f_n,\qquad S=\sumn \mu_n g_n\otimes g_n
	\end{equation*}
	for some orthonormal sets $\seqn{f}$ and $\seqn{g}$ in $L^2$ and $\lambda_n,\mu_n\geq 0$. Let $\seqn{e}$ be an o.n.b. for $L^2(\rd)$:
	\begin{equation*}
		\tr(TS)=\sumn\la TS e_n,e_n\ra=\sum_{i,j}^\infty \lambda_j \mu_i \abs{\la f_j,g_i\ra}^2. 
	\end{equation*}
	On the other hand, 
	\begin{equation*}
		\intrdd \tW T(z)\overline{\tW S(z)}\,dz=\sum_{i,j}^\infty\lambda_j\mu_i\intrdd\tW f_j(z)\overline{\tW g_i(z)}\,dz=\sum_{i,j}^\infty \lambda_j \mu_i \abs{\la f_j,g_i\ra}^2,
	\end{equation*}
	where the last equality is due to Moyal's identity. This concludes the proof.
\end{proof}

\begin{remark}
	Since we assume $S\geq 0$, $S$ is self-adjoint and for $\tau=1/2$ we have that $\ohW S$ is real-valued. In fact, using the representation given in the proof of Lemma \ref{Lem-H.R.-Lem-2.7}:
	\begin{equation*}
		\ohW S=\sumn \mu_n\ohW g_n
	\end{equation*}
	with every $\ohW g_n$ real-valued and $\mu_n\geq 0$. Hence, for $\tau=1/2$ we recover \cite[Lemma 2.7]{HerRei2022}.
\end{remark}

\begin{lemma}\label{Lem-H.R.-Lem-2.9}
	Let $T\in\TC$ and consider $\f\in\Sch(\rd)$ such that $\norm{\f}_{L^2}=1$. Then
	\begin{equation}
		\tr T=\intrdd \la (\alpha_z T)\f,\f\ra\,dz=\intrdd Q_T\f(z)\,dz=\intrdd \GM{T}(z,z)\,dz.
	\end{equation}
\end{lemma}
\begin{proof}
	The proof follows from a direct computation using the representations presented in the proof of Lemma \ref{Lem-H.R.-Lem-2.7} and Moyal's identity involving the function $\f$:
	\begin{equation*}
		\la f_j,g_i\ra=\la V_{\f}f_j,V_{\f}g_i\ra,
	\end{equation*}
	we leave details to the interested reader.
\end{proof}

\begin{lemma}\label{Lem-H.R.-Lem-2.10}
	Let $T\in\TC$, $T\geq 0$ and let $\f\in\Sch(\rd)$ such that $\norm{\f}_{L^2}=1$. Then for every $z\in\rdd$:
	\begin{equation}
		Q_T\f(z)=\intrdd\tW T(w)\overline{\tW \f(z+w)}\,dw=\tW T\ast (\tW\f)^\ast(z),
	\end{equation}
	where $( \tW\f)^\ast(w)=\overline{\tW\f(-w)}$.
\end{lemma}
\begin{proof}
	We compute directly
	\begin{align*}
		Q_T\f(z)&=\la\pi(z)T\pi(z)^\ast\f,\f\ra=\tr(T(\pi(z)^\ast\f\otimes\pi(z)^\ast\f))\\
		&=\intrdd\tW T(w)\overline{\tW(\pi(z)^\ast\f\otimes\pi(z)^\ast\f)(w)}\,dw,
	\end{align*}
	the last equation holds because of Lemma \ref{Lem-H.R.-Lem-2.7}. An elementary calculation gives
	\begin{equation*}
		\tW(\pi(z)^\ast\f\otimes\pi(z)^\ast\f)(w)=\tW \f(z+w),
	\end{equation*}
	which is also known as covariance property and this concludes the proof.
\end{proof}

\begin{lemma}\label{Lem-H.R.-Lem-3.1}
	Let $T\in\TC$, $T\geq 0$ and consider $\f\in\Sch(\rd)$ such that $\norm{\f}_{L^2}=1$. Then for every $z,w\in\rdd$:
	\begin{equation*}
		\abs{\GM{T}(z,w)}^2\leq Q_T\f(-z) Q_T\f (-w).
	\end{equation*}
\end{lemma}
\begin{proof}
	The claim follows from the Cauchy-Schwarz inequality for the inner product induced by the positive operator $T$ and Remark \ref{Rem-GaborMatrix-CohensClass}.
\end{proof}


\begin{lemma}\label{Lem-H.R.-Lem-3.3}
	Let $0_d$ and $I_d$ denote the zero and identity $d\times d$ matrices, respectively. Let us define 
	\begin{equation*}
		\twistmatrix\coloneqq
		\begin{bmatrix}
			0_d&0_d&0_d&0_d\\
			I_d&0_d&0_d&0_d\\
			0_d&0_d&0_d&0_d\\
			0_d&0_d&-I_d&0_d
		\end{bmatrix}.
	\end{equation*}
	Let $T\in\TC$ and consider $\f\in\Sch(\rd)$ such that $\norm{\f}_{L^2}=1$. For $z=(x,\om),w=(u,v)\in\rdd$ we have
	\begin{align}
		\GM{T}(z,w)&=\GM{T}\twistconv(\GM{\f\otimes\f})^\ast(z,w)\\
		&=\intrdd\intrdd\GM{T}(z',w')(\GM{\f\otimes\f})^\ast(z-z',w-w')e^{2\pi i (\om x'-u' v)}\,dz'dw',\notag
	\end{align}
	where $z'=(x',\om'),w'=(u',v')\in\rdd$.
\end{lemma}
\begin{proof}
	We apply twice Moyal's identity:
	\begin{align*}
		\GM{T}(z,w)&=\intrdd V_\f[T\pi(w)\f](z')\overline{V_\f[\pi(z)\f](z')}\,dz'\\
		&=\intrdd\intrdd V_\f[\pi(w)\f](w')\overline{V_\f[T^\ast \pi(z')\f](w')}\la\pi(z')\f,\pi(z)\f\ra\,dz'dw'\\
		&=\intrdd\intrdd \GM{T}(z',w')\la\pi(w)\f,\pi(w')\f\ra\la\pi(z')\f,\pi(z)\f\ra\,dz'dw'.
	\end{align*}
	It is then a direct, although tedious, calculation to show that
	\begin{equation*}
		\la\pi(z)\f,\pi(z')\f\ra\la\pi(w')\f,\pi(w)\f\ra=(\GM{\f\otimes\f})^\ast(z-z',w-w')e^{2\pi i (\om x'-u' v)}.
	\end{equation*}
	This concludes the proof.
\end{proof}


\begin{lemma}\label{Lem-H.R.-Lem-3.5}
	Let $T\in\TC$, $T\geq 0$ and consider $\f\in\Sch(\rd)$ such that $\norm{\f}_{L^2}=1$.  Then for any $\tau\in[0,1]$:
	\begin{equation}
		\tW T(z)=\intrdd\intrdd e^{-2\pi i[(\om x'-\om' x)+(\frac12-\frac34\tau)x'\om'+x'v]}\GM{T}\left(\frac{z'}{2}-w,-\frac{z'}{2}-w\right)\,dwdz',
	\end{equation}
	where $z=(x,\om),z'=(x',\om'),w=(u,v)\in\rdd$.
\end{lemma}
\begin{proof}
	We start rephrasing the $\tau$-Wigner distribution of $T$:
	\begin{equation*}
		\tW T(z)=\cF_{\sigma}\FtW T(z)=\intrdd e^{-2\pi i (\om x'-\om' x)}\tr(\tTFS(z')^\ast T)\,dz'.
	\end{equation*}
	Recalling the properties for $\tTFS$, see Section \ref{Sec-Preliminaries}, we see that
	\begin{align*}
		\tTFS(z'/2+z'/2)&=e^{2\pi i[(1-\tau)\frac{x'\om'}{4}-\tau \frac{x'\om'}{4}]}\tTFS(z'/2)\tTFS(z'/2)\\
		&=e^{\frac{\pi}{2}i(1-2\tau)x'\om'}\tTFS(z'/2)\tTFS(z'/2).
	\end{align*}
	Taking the adjoint we get $\tTFS(z')^\ast=e^{-\frac{\pi}{2}i(1-2\tau)x'\om'}\tTFS(z'/2)^\ast\tTFS(z'/2)^\ast$ and we write using Lemma \ref{Lem-H.R.-Lem-2.9}:
	\begin{align*}
		\tr(\tTFS(z')^\ast T)&=e^{-\frac{\pi}{2}i(1-2\tau)x'\om'}\tr(\tTFS(z'/2)^\ast T\tTFS(z'/2)^\ast)\\
		&=e^{-\frac{\pi}{2}i(1-2\tau)x'\om'}\intrdd\la T\tTFS(z'/2)^\ast\tTFS(w)^\ast\f,\tTFS(z'/2)\tTFS(w)^\ast\f\ra\,dw\\
		&=e^{-\frac{\pi}{2}i(1-2\tau)x'\om'}e^{-\frac{\pi}{2}i(1-\tau)x'\om'}\\
		&\times \intrdd \la T\tTFS(-z'/2)\tTFS(-w)\f,\tTFS(z'/2)\tTFS(-w)\f\ra\,dw\\
		&=e^{-\frac{\pi}{2}i(2-3\tau)x'\om'}\intrdd \la T\pi(-z'/2)\pi(-w)\f,\pi(z'/2)\pi(-w)\f\ra\,dw\\
		&=e^{-\frac{\pi}{2}i(2-3\tau)x'\om'}\intrdd e^{-2\pi i x'v} \la T\pi(-z'/2-w)\f,\pi(z'/2-w)\f\ra\,dw.
	\end{align*}
	This concludes the argument.
\end{proof}

%
%
\section{A characterization of Schwartz operators}\label{Ch-7-Sec-char-Sch-op}
In this section we introduce weighted versions of $\FO$ and give an alternative description of the class $\mathfrak{S}$. We use the polynomial weight
\begin{equation}\label{Eq-v_s}
	v_s(z)\coloneqq (1+\abs{z}^2)^{\frac{s}{2}},\qquad z\in\rdd,
\end{equation}
where $s\geq 0$. In order to avoid an extremely cumbersome notation, just for the weight  functions $v_s$ we shall use the following: 
\begin{equation*}
	v_s\otimes v_s(z,w)\coloneqq \K{v_s\otimes \overline{v_s}}=v_s(z)v_s(w),\qquad\forall z,w\in\rdd.
\end{equation*}

\begin{definition}
	For $s\geq 0$ we define the weighted class of Feichtinger operators as
	\begin{equation}
		\wFO\coloneqq\{S\colon \MD(\rd)\to\FA(\rd)\,|\,S\,\text{is linear, continuous with kernel}\,\K{S}\in\wFA(\rdd)\}.
	\end{equation}
	For $S$ in $\wFO$ we define the mapping 
	\begin{equation}\label{Eq-norm-wFO}
		\norm{S}_{\wFO}\coloneqq\norm{\K{S}}_{\wFA}.
	\end{equation}
\end{definition}

\begin{remark}
	\begin{itemize}
		\item[$(i)$] For $s=0$ we recover the Feichtinger operators $\FO$;
		\item[$(ii)$] The mapping defined in \eqref{Eq-norm-wFO} is a norm on $\wFO$ and it is easy to see that $(\wFO,\norm{\cdot}_{\wFO})$ is a Banach space and the following continuous inclusion holds true for every $s\geq0$:
		\begin{equation}
			\wFO\hookrightarrow\FO.
		\end{equation}
	\end{itemize}
\end{remark}

\begin{lemma}
	For any $S\in\wFO$ there exist $\seqn{f},\seqn{g}\subseteq \wFA(\rdd)$ such that
	\begin{equation*}
		S=\sumn f_n\otimes g_n,\qquad\sumn \norm{f_n}_{M^1_{v_s}}\norm{g_n}_{M^1_{v_s}}\leq+\infty,\quad \K{S}=\sumn\K{f_n\otimes g_n}.
	\end{equation*}
\end{lemma}
\begin{proof}
	The proof follows from the fact that
	\begin{equation*}
		\wFA(\rdd)=M^1_{v_s}(\rd)\hat{\otimes}M^1_{v_s}(\rd).
	\end{equation*}
	See also the proof of Lemma \ref{Lem-representation-FO}.
\end{proof}

\begin{theorem}
	For every $\tau\in[0,1]$ the mapping $\tW\colon\wFO\to\wFA(\rdd)$ is a topological isomorphism with inverse given by $\Opt\colon\wFA(\rdd)\to\wFO$.
\end{theorem}
\begin{proof}
	The proof follows the same pattern as the ones of Theorem \ref{Th-Opt-tW-FO-MD} and Corollary \ref{Cor-tW-Opt-iso-FO-FA}. 
\end{proof}

\begin{corollary}\label{Cor-char-wFO}
	An operator $S$ belongs to $\wFO$ if and only if for some (hence every) $\tau\in[0,1]$ $\tW S\in\wFA(\rdd)$.
\end{corollary}

\begin{theorem}\label{Th-char-Sch-Op}
	The following is true:
	\begin{equation}
		\mathfrak{S}=\bigcap_{s\geq 0}\wFO.
	\end{equation}
\end{theorem}
\begin{proof}
	By Corollary \ref{Cor-char-wFO}, $S$ belongs to the set on the right-hand side if and only if 
	\begin{equation*}
		\tW S\in\bigcap_{s\geq 0}\wFA(\rdd)=\cS(\rdd).
	\end{equation*}
	The claim follows since $\ohW S$ is the Weyl symbol of $S$, i.e. $\mathrm{a}_{1/2}^S=\ohW S$.
\end{proof}

We recall that a function $F$ on $\rdd$ is called rapidly decaying if for every multiindex $\alpha,\beta\in\bN_0^{d}$ we have
\begin{equation*}
	\sup_{x,\om\in\rd}\abs{x^\al \om^\be F(x,\om)}<+\infty,
\end{equation*}
where, if $x=(x_1,\ldots,x_d)$ and $\alpha=(\alpha_1,\ldots,\alpha_d)$, $x^\al$ stands for $x_1^{\alpha_1}\cdot\ldots\cdot x_d^{\alpha_d}$.\\

In \cite[Theorem 1.1]{HerRei2022} a sufficient condition is given for a positive trace-class operator to be in $\mathfrak{S}$. Namely, if $T\in  B(L^2)$, $T\geq 0$, is such that $\tW T$ exists for some $\tau \in[0,1]$ and it is rapidly decreasing, then $T\in\mathfrak{S}$ and $\tW T$ exists for every $\tau\in[0,1]$. In this spirit, we provide the following sufficient condition for a generic $S\in B(L^2)$. Observe that we do not not require $S$ to be positive.

\begin{corollary}\label{Cor-suff-cond-Sch-op}
	Let $S\in B(L^2)$ and assume that for some $\tau\in[0,1]$ $\tW S$ exists. Suppose also that, w.r.t. some non-zero window in  $L^2(\rdd)$, the STFT of $\tW S$ is rapidly decaying. Then $\tW S$ exists for every $\tau\in [0,1]$ and $S$ is in $\mathfrak{S}$. 
\end{corollary}
\begin{proof}
	Let us pick $G\in L^2(\rdd)\smallsetminus\{0\}$. If $V_G \tW{S}$ is rapidly decaying then $S\in\wFO$ for every $s\geq 0$. The claim follows from Theorem \ref{Th-char-Sch-Op}.
\end{proof}

\section*{Acknowledgments}  
The first author would like to thank Eduard Ortega for the financial support to visit Trondheim which led to this work.


\begin{thebibliography}{99}
	






\bibitem{BasCorNic20}
F.~Bastianoni, E.~Cordero and F.~Nicola.
\newblock Decay and smoothness for eigenfunctions of localization operators.
\newblock {\em J. Math. Anal. Appl.} \textbf{492}, 124480, 2020.








\bibitem{Olebook2003}
O.~Christensen.
\newblock {\em An introduction to frames and {R}iesz bases}.
Applied and Numerical Harmonic Analysis,
Birkh\"{a}user Basel, Second Edition, 2016.




	\bibitem{EleCharly2003}
	E.~Cordero and K.~Gr\"ochenig.
	\newblock Time-frequency analysis of localization operators.
	\newblock {\em J. Funct. Anal.}, 205(1):107--131, 2003.


\bibitem{Wignersharp2018}
E.~Cordero and F.~Nicola.
\newblock Sharp integral bounds for {W}igner distributions.
\newblock {\em Int. Math. Res. Not. IMRN}, (6):1779--1807, 2018.


\bibitem{CordRod2020}
E.~Cordero and L.~Rodino.
\newblock {\em Time-Frequency analysis of operators}.
\newblock De Gruyter Studies in Mathematics 75, Berlin/Boston, 2020.



\bibitem{DoeLueMcnSkr2022}
M.~D{\"o}rfler, F.~Luef, H.~McNulty and E.~Skrettingland.
\newblock Time-Frequency Analysis and Coorbit Spaces of Operators.
\newblock{\em arXiv preprint arXiv:2210.04844}, 2022.




\bibitem{GosGos2021}
C. de~Gosson and  M. de~Gosson.
\newblock On the Non-Uniqueness of Statistical Ensembles Defining a Density Operator and a Class of Mixed Quantum States with Integrable Wigner Distribution.
\newblock {\em Quantum Reports}, 3(3):473-81, 2021.

\bibitem{deVries1978}
J.~De Vries.
\newblock The local weight of an effective locally compact transformation group and the dimension og $L^2(G)$.
\newblock {\em Colloq. Math.} 39(2): 319--3323, 1978.





\bibitem{Feichtinger-On-a-new-segal-alg-1981}
H.~G. Feichtinger.
\newblock On a new Segal algebra.
\newblock {\em Monatshefte f\"{u}r Mathematik} {\bf 92}, 269--289, 1981.










\bibitem{FeiJak2022}
H.~G. Feichtinger and M.~S. Jakobsen.
\newblock The inner kernel theorem for a certain Segal algebra.
\newblock {\em Monatsh. Math.}, 2022.



%
%














\bibitem{GroStr2007}
K.~Gr\"{o}chenig and T.~Strohmer.
\newblock Pseudodifferential operators on locally compact abelian groups and Sj\"{o}strand's symbol class.
\newblock {\em Journal für die reine und angewandte Mathematik}, 2007(613), 121--146, 2007.







\bibitem{HerRei2022}
F.~Hern\'andez and C.~J.~Riedel.
\newblock Rapidly decaying Wigner functions are Schwartz functions.
\newblock {\em J. Math. Phys.} \textbf{63}, 022104, 2022.




\bibitem{Jakobsen2018}
M.~S.~Jakobsen.
\newblock  On a (no longer) new Segal algebra: a review of the Feichtinger algebra.
\newblock {\em J. Fourier Anal. Appl.}, 24:1579--1660, 2018.

\bibitem{KeyKiuWer2016}
M.~Keyl, J.~Kiukas and R.~Werner.
\newblock Schwartz operators.
\newblock {\em Rev. Math. Phys.} \textbf{28}(3), 1630001, 60, 2016. 

\bibitem{Laf2022}
L.~Lafleche.
\newblock On Quantum Sobolev Inequalities.
\newblock{\em arXiv preprint arXiv:2210.03013}, 2022.

%


\bibitem{Luef2}
F.~Luef and E.~Skrettingland.
\newblock On accumulated {C}ohen's class distributions and mixed-state
localization operators.
\newblock {\em Constr. Approx.} \textbf{52}, 31--64, 2020.

\bibitem{Luef1}
F.~Luef and E.~Skrettingland.
\newblock Mixed-state localization operators: {C}ohen's class and trace
class operators.
\newblock {\em J. Fourier Anal. Appl.}, 25(4):2064--2108, 2019.



	\bibitem{Megginson}
R.~Megginson.
\newblock {\em An Introduction to Banach Space Theory}.
\newblock Graduate Texts in Mathematics, vol.183,
pp. xx+596. Springer, New York, 1998.


\bibitem{Moy1949}
J.~E.~Moyal.
\newblock Quantum mechanics as a statistical theory.
\newblock {\em Proc. Cambridge Phil. Soc.}, 45:99--124, 1949.















\bibitem{Eirik-Master}
E.~Skrettingland.
\newblock {\em Convolutions for Localization Operators}.
\newblock Master Thesis, NTNU, 2017.


\bibitem{SimonTrace}
B.~Simon. 
\newblock {\em Trace Ideal and Their Applications}.
\newblock Cambridge University Press, Cambridge, 1979.










\bibitem{ToftquasiBanach2017}
J.~Toft.
\newblock Continuity and compactness for pseudo-differential operators
with symbols in quasi-{B}anach spaces or {H}\"{o}rmander classes.
\newblock {\em Anal. Appl. (Singap.)}, 15(3):353--389, 2017.







\bibitem{Werner1984}
R.~F. Werner.
\newblock Quantum harmonic analysis on phase space.
\newblock {\em J. Math. Phys.} \textbf{25}(5), 1404--1411, 1984.

\bibitem{Wigner1932}
E.~P. Wigner.
\newblock On the quantum correction for thermodynamic equilibrium.
\newblock {\em Phys. Rev.} \textbf{40}, 749--759, 1932.




	
\end{thebibliography}
\end{document}